\documentclass[12pt,reqno]{amsart}

\usepackage[sanserif,full]{complexity}

\usepackage{graphicx}
\usepackage{amssymb}
\usepackage{amsmath}
\usepackage{amsthm}

\newtheorem{thm}{Theorem}
\newtheorem{defn}[thm]{Definition}
\newtheorem{lem}[thm]{Lemma}

\newtheorem{prob}[thm]{Problem}

\begin{document}

\title{Full Spark Frames}

\author[Alexeev]{Boris Alexeev}
\address[Alexeev]{Department of Mathematics, Princeton University, Princeton, New Jersey 08544, USA; E-mail: balexeev@math.princeton.edu}

\author[Cahill]{Jameson Cahill}
\address[Cahill]{Department of Mathematics, University of Missouri, Columbia, Missouri 65211, USA; E-mail: jameson.cahill@gmail.com}

\author[Mixon]{Dustin G.~Mixon}
\address[Mixon]{Program in Applied and Computational Mathematics, Princeton University, Princeton, New Jersey 08544, USA; E-mail: dmixon@princeton.edu}

\date{}



\begin{abstract}
Finite frame theory has a number of real-world applications.
In applications like sparse signal processing, data transmission with robustness to erasures, and reconstruction without phase, there is a pressing need for deterministic constructions of frames with the following property: every size-$M$ subcollection of the $M$-dimensional frame elements is a spanning set.
Such frames are called full spark frames, and this paper provides new constructions using the discrete Fourier transform.
Later, we prove that full spark Parseval frames are dense in the entire set of Parseval frames, meaning full spark frames are abundant, even if one imposes an additional tightness constraint.
Finally, we prove that testing whether a given matrix is full spark is hard for $\NP$ under randomized polynomial-time reductions, indicating that deterministic full spark constructions are particularly significant because they guarantee a property which is otherwise difficult to check.
\end{abstract}

\thanks{The authors thank Profs.~Peter G.~Casazza and Matthew Fickus for discussions on full spark frames, Prof. Dan Edidin and Will Sawin for discussions on algebraic geometry, and the anonymous referees for very helpful comments and suggestions.
BA was supported by the NSF Graduate Research Fellowship under Grant No.~DGE-0646086, JC was supported by NSF Grant No.~DMS-1008183, DTRA/NSF Grant No.~DMS-1042701 and AFOSR Grant No.~FA9550-11-1-0245, and DGM was supported by the A.B.~Krongard Fellowship. The views expressed in this article are those of the authors and do not reflect the official policy or position of the United States Air Force, Department of Defense, or the U.S. Government.
}

\maketitle

\section{Introduction}

A \emph{frame} over a Hilbert space $\mathcal{H}$ is a collection of vectors $\{f_i\}_{i\in\mathcal{I}}\subseteq\mathcal{H}$ with \emph{frame bounds} $0<A\leq B<\infty$ such that
\begin{equation}
\label{eq.frame defn}
A\|x\|^2\leq\sum_{i\in\mathcal{I}}|\langle x,f_i\rangle|^2\leq B\|x\|^2
\qquad
\forall x\in\mathcal{H}.
\end{equation}
For a finite frame $\{f_n\}_{n=1}^N$ in an $M$-dimensional Hilbert space $\mathbb{H}_M$, the upper frame bound $B$ which satisfies \eqref{eq.frame defn} trivially exists.
In this finite-dimensional setting, the notion of being a frame solely depends on the lower frame bound $A$ being strictly positive, which is equivalent to having the frame elements $\{f_n\}_{n=1}^N$ span the Hilbert space.

In practice, frames are chiefly used in two ways.
The \emph{synthesis operator} $F\colon\mathbb{C}^N\rightarrow\mathbb{H}_M$ of a finite frame $F=\{f_n\}_{n=1}^N$ is defined by $Fy:=\sum_{n=1}^Ny[n]f_n$.
As such, the $M\times N$ matrix representation $F$ of the synthesis operator has the frame elements $\{f_n\}_{n=1}^N$ as columns, with the natural identification $\mathbb{H}_M\cong\mathbb{C}^M$.
Note that here and throughout, we make no notational distinction between a frame $F$ and its synthesis operator.
The adjoint of the synthesis operator is called the \emph{analysis operator} $F^*\colon\mathbb{H}_M\rightarrow\mathbb{C}^N$, defined by $(F^*x)[n]:=\langle x,f_n\rangle$.

Classically, frames are used to redundantly decompose signals, $y=F^*x$, before synthesizing the corresponding frame coefficients, $z=Fy=FF^*x$, and so the \emph{frame operator} $FF^*\colon\mathbb{H}_M\rightarrow\mathbb{H}_M$ is often analyzed to determine how well this process preserves information about the original signal $x$.
In particular, if the frame bounds are equal, the frame operator has the form $FF^*=AI_M$, and so signal reconstruction is rather painless: $x=\frac{1}{A}FF^*x$; in this case, the frame is called \emph{tight}.
Oftentimes, it is additionally desirable for the frame elements to have unit norm, in which case the frame is a \emph{unit norm frame}.
Moreover, the \emph{worst-case coherence} between unit norm frame elements $\mu:=\max_{i\neq j}|\langle f_i,f_j\rangle|$ satisfies $\mu^2\geq\frac{N-M}{M(N-1)}$, and equality is achieved precisely when the frame is tight with $|\langle f_i,f_j\rangle|=\mu$ for all distinct pairs $i,j\in\{1,\ldots,N\}$~\cite{StrohmerH:acha03}; in this case, the frame is called an \emph{equiangular tight frame} (ETF).

The utility of unit norm tight frames and ETFs is commonly expressed in terms of a scenario in which frame coefficients $\{(F^*x)[n]\}_{n=1}^N$ are transmitted over a noisy or lossy channel before reconstructing the signal:
\begin{equation}
\label{eq.reconstruction}
y=\mathcal{D}(F^*x),\qquad \tilde{x}=(FF^*)^{-1}Fy,                                                                                                                                                                                                                  \end{equation}
where $\mathcal{D}(\cdot)$ represents the channel's random and not-necessarily-linear deformation process.
Using an additive white Gaussian noise model, Goyal~\cite{Goyal:phd98} established that, of all unit norm frames, unit norm tight frames minimize mean squared error in reconstruction.
For the case of a lossy channel, Holmes and Paulsen~\cite{HolmesP:laa04} established that, of all tight frames, unit norm tight frames minimize worst-case error in reconstruction after one erasure, and that equiangular tight frames minimize this error after two erasures.
We note that the reconstruction process in \eqref{eq.reconstruction}, namely the application of $(FF^*)^{-1}F$, is inherently blind to the effect of the deformation process of the channel.
This contrasts with P\"{u}schel and Kova\v{c}evi\'{c}'s more recent work~\cite{PuschelK:dcc05}, which describes an adaptive process for reconstruction after multiple erasures; we will return to this idea later.

Another application of frames is sparse signal processing.
This field concerns signal classes which are sparse in some basis.
As an example, natural images tend to be nearly sparse in the wavelet basis~\cite{DavenportDEK:csta11}.
Some applications have signal sparsity in the identity basis~\cite{MixonQKF:icassp11}.
Given a signal $x$, let $\Psi$ represent its sparsifying basis and consider
\begin{equation}
\label{eq.sparse signal processing}
y=F\Psi^* x+e,
\end{equation}
where $F$ is the $M\times N$ synthesis operator of a frame, $\Psi^* x$ has at most $K$ nonzero entries, and $e$ is some sort of noise.
When given measurements $y$, one typically wishes to reconstruct the original vector $x$; viewing $F$ as a sensing matrix, this process is commonly referred to as \emph{compressed sensing} since we often take $M\ll N$.
Note that $y$ can be viewed as a noisy linear combination of a few unknown frame elements, and the goal of compressed sensing is to determine which frame elements are active in the combination, and further estimate the scalar multiples used in this combination.
This problem setup is very related to that of sparse approximation, in which signals are known to be expressible as a sparse combination of elements from an \emph{overcomplete} dictionary; in this case, the dictionary elements form the columns of $F$, and we again wish to determine the scalar multiples used in a given sparse combination.

In order to have any sort of inversion process for \eqref{eq.sparse signal processing}, even in the noiseless case where $e=0$, $F$ must map $K$-sparse vectors injectively.
That is, we need $Fx_1\neq Fx_2$ for any distinct $K$-sparse vectors $x_1$ and $x_2$.
Subtraction then gives that $2K$-sparse vectors like $x_1-x_2$ cannot be in the nullspace of $F$.
This identification led to the following definition~\cite{DonohoE:pnas03}:

\begin{defn}
The \emph{spark} of a matrix $F$ is the size of the smallest linearly dependent subset of columns, i.e.,
\begin{equation*}
\mathrm{Spark}(F)
=\min\{\|x\|_0:Fx=0,~x\neq0\}.
\end{equation*}
\end{defn}

Using the above analysis, Donoho and Elad~\cite{DonohoE:pnas03} showed that a signal $x$ with sparsity level $<\mathrm{Spark}(F)/2$ is necessarily the unique sparsest solution to $y=Fx$, and furthermore, there exist signals $x$ of sparsity level $\geq\mathrm{Spark}(F)/2$ which are not the unique sparsest solution to $y=Fx$.
This demonstrates that matrices with larger spark are naturally equipped to support signals with larger sparsity levels.
One is naturally led to consider matrices that support the largest possible sparse signal class; we say an $M\times N$ matrix $F$ is \emph{full spark} if its spark is as large as possible, i.e., $\mathrm{Spark}(F)=M+1$.
Equivalently, $M\times N$ full spark matrices have the property that every $M\times M$ submatrix is invertible; as such, a full spark matrix is necessarily full rank, and therefore a frame.

That being said, while the submatrices of a full spark matrix will be invertible, they may not be well-conditioned.
Moreover, the conditioning of these submatrices is an important feature for compressed sensing; Cand\`{e}s~\cite{Candes:08} gives an elegant proof that sensing matrices with well-conditioned submatrices, specifically, satisfying the \emph{restricted isometry property (RIP)}, allow for stable and efficient recovery.
Unfortunately, for deterministic matrices, it is difficult to determine the conditioning of every submatrix; to date, no deterministic RIP matrix is known to perform optimally~\cite{FickusM:spie11}, and in one respect, this difficulty is precisely the barrier to significant progress on the Kadison-Singer problem~\cite{CasazzaT:06}.
However, some work has been done to test whether \emph{most} submatrices are well-conditioned~\cite{Tropp:acha08}, and better yet, other work has managed to achieve RIP-like performance without requiring all submatrices to be well-conditioned~\cite{BajwaCM:acha11}.
Regardless, good sensing matrices for compressed sensing must necessarily have large spark to enable reconstruction~\cite{DonohoE:pnas03}, and so building such matrices could serve as one step toward optimal deterministic RIP matrices.

In sparse signal processing, the specific application of full spark frames has been studied for some time.
In 1997, Gorodnitsky and Rao~\cite{GorodnitskyR:97} first considered full spark frames, referring to them as matrices with the \emph{unique representation property} (for reasons discussed above).
Since~\cite{GorodnitskyR:97}, the unique representation property has been explicitly used to find a variety of performance guarantees for sparse signal processing~\cite{BourguignonCI:07,MohimaniBJ:09,WipfR:04}.
Tang and Nehorai~\cite{TangN:10} also obtain performance guarantees using full spark frames, but they refer to them as \emph{non-degenerate measurement matrices}.

Consider a matrix whose entries are independent continuous random variables.
Intuitively, the matrix is full spark with probability one, and this fact was recently proved in~\cite{BlumensathD:09}.
However, this random process does not allow one to control certain features of the matrix, such as unit-norm tightness or equiangularity.
Indeed, ETFs are notoriously difficult to construct, but they appear to be particularly well-suited for sparse signal processing~\cite{BajwaCM:acha11,FickusM:spie11,MixonQKF:icassp11}.
For example, Bajwa et al.~\cite{BajwaCM:acha11} uses Steiner ETFs to recover the support of $\Psi^* x$ in \eqref{eq.sparse signal processing}; given measurements $y$, the largest entries in the back-projection $F^*y$ often coincide with the support of $\Psi^* x$.
However, Steiner ETFs have particularly small spark~\cite{FickusMT:laa11}, and back-projection will correctly identify the support of $x$ even when the corresponding columns in $F$ are dependent; in this case, there is no way to estimate the nonzero entries of $\Psi^* x$.
This illustrates one reason to build deterministic full spark frames: their submatrices are necessarily invertible, making it possible to estimate these nonzero entries.

For another application of full spark frames, we return to the problem \eqref{eq.reconstruction} of reconstructing a signal from distorted frame coefficients.
Specifically, we consider P\"{u}schel and Kova\v{c}evi\'{c}'s work~\cite{PuschelK:dcc05}, which focuses on an Internet-like channel that is prone to multitudes of erasures.
In this context, they reconstruct the signal after first identifying which frame coefficients were not erased; with this information, the signal can be estimated provided the corresponding frame elements span.
In this sense, full spark frames are \emph{maximally robust to erasures}, as coined in~\cite{PuschelK:dcc05}.
In particular, an $M\times N$ full spark frame is robust to $N-M$ erasures since any $M$ of the frame coefficients will uniquely determine the original signal.

Yet another application of full spark frames is phaseless reconstruction, which can be viewed in terms of a channel, as in \eqref{eq.reconstruction}; in this case, $\mathcal{D}(\cdot)$ is the entrywise absolute value function.
Phaseless reconstruction has a number of real-world applications including speech processing~\cite{BalanCE:acha06}, X-ray crystallography~\cite{CandesSV:arxiv11}, and quantum state estimation~\cite{RenesBSC:jmp04}.
As such, there has been a lot of work to reconstruct an $M$-dimensional vector (up to an overall phase factor) from the magnitudes of its frame coefficients, most of which involves frames in operator space, which inherently require $N=\Omega(M^2)$ measurements~\cite{BalanBCE:jfaa09,RenesBSC:jmp04}.
However, Balan et al.~\cite{BalanCE:acha06} show that if an $M\times N$ real frame $F$ is full spark with $N\geq 2M-1$, then $\mathcal{D}\circ F^*$ is injective, meaning an inversion process is possible with only $N=O(M)$ measurements.
This result prompted an ongoing search for efficient phaseless reconstruction processes~\cite{BalanBCE:spie07,CandesSV:arxiv11}, but no reconstruction process can succeed without a good family of frames, such as full spark frames.

Despite the fact that full spark frames have a multitude of applications, to date, there has not been much progress in constructing deterministic full spark frames, let alone full spark frames with additional desirable properties.
A noteworthy exception is P\"{u}schel and Kova\v{c}evi\'{c}'s work~\cite{PuschelK:dcc05}, in which real full spark tight frames are constructed using polynomial transforms.
In the present paper, we start by investigating Vandermonde frames, harmonic frames, and modifications thereof.
While the use of certain Vandermonde and harmonic frames as full spark frames is not new~\cite{BourguignonCI:07,CandesRT:06,Fuchs:05}, the fruits of our investigation are new:
For instance, we demonstrate that certain classes of ETFs are full spark, and we characterize the $M\times N$ full spark harmonic frames for which $N$ is a prime power.
The remainder of the paper proves two results which might be considered folklore-type observations---while their proofs are new, the results are not surprising.
First, in Section~3, we show that $M\times N$ full spark Parseval frames form a dense subset of the entire collection of $M\times N$ Parseval frames.
As such, full spark frames are abundant, even after imposing the additional condition of tightness.
This result is balanced with Section~4, in which we prove that verifying whether a matrix is full spark is hard for $\NP$ under randomized polynomial-time reductions.
In other words, assuming $\NP\not\subseteq\BPP$ (a computational complexity assumption slightly stronger than $\P\neq\NP$ and nearly as widely believed), then there is no method by which one can efficiently test whether matrices are full spark.
As such, the deterministic constructions in Section~2 are significant in that they guarantee a property which is otherwise difficult to check.

\section{Deterministic constructions of full spark frames}

A square matrix is invertible if and only if its determinant is nonzero, and in our quest for deterministic constructions of full spark frames, this characterization will reign supreme.
One class of matrices has a particularly simple determinant formula: Vandermonde matrices.
Specifically, Vandermonde matrices have the following form:
\begin{equation}
\label{eq.vandermonde}
V
=\begin{bmatrix}
 1 & 1 & \cdots & 1\\
 \alpha_1 & \alpha_2 & \cdots & \alpha_N\\
\vdots & \vdots & \cdots & \vdots\\
 \alpha_1^{M-1} & \alpha_2^{M-1} & \cdots & \alpha_N^{M-1}
 \end{bmatrix},
\end{equation}
and square Vandermonde matrices, i.e., with $N=M$, have the following determinant:
\begin{equation}
\label{eq.vandermonde determinant}
\mathrm{det}(V)=\prod_{1\leq i<j\leq M}(\alpha_j-\alpha_i).
\end{equation}
Consider \eqref{eq.vandermonde} in the case where $N\geq M$.
Since every $M\times M$ submatrix of $V$ is also Vandermonde, we can modify the indices in \eqref{eq.vandermonde determinant} to calculate the determinant of the submatrices.  
These determinants are nonzero precisely when the bases $\{\alpha_{n}\}_{n=1}^N$ are distinct, yielding the following result:

\begin{lem}
\label{lem.vandermonde}
A Vandermonde matrix is full spark if and only if its bases are distinct.
\end{lem}

To be clear, this result is not new.
In fact, the full spark of Vandermonde matrices was first exploited by Fuchs~\cite{Fuchs:05} for sparse signal processing.
Later, Bourguignon et al.~\cite{BourguignonCI:07} specifically used the full spark of Vandermonde matrices whose bases are sampled from the complex unit circle.
Interestingly, when viewed in terms of frame theory, Vandermonde matrices naturally point to the discrete Fourier transform:

\begin{thm}
\label{thm.vandermonde}
The only $M\times N$ Vandermonde matrices that are equal norm and tight have bases in the complex unit circle.  Among these, the frames with the smallest worst-case coherence have bases that are equally spaced in the complex unit circle, provided $N\geq 2M$.
\end{thm}

\begin{proof}
Suppose a Vandermonde matrix is equal norm and tight.  
Note that a zero base will produce the zeroth identity basis element $\delta_0$.
Letting $\mathcal{P}$ denote the indices of the nonzero bases, the fact that the matrix is full rank implies $|\mathcal{P}|\geq M-1$.
Also, equal norm gives that the frame element length
\[ \|f_n\|^2=\sum_{m=0}^{M-1}|f_n[m]|^2=\sum_{m=0}^{M-1}|\alpha_n^m|^2=\sum_{m=0}^{M-1}|\alpha_n|^{2m} \]
is constant over $n\in\mathcal{P}$.  
Since $\sum_{m=0}^{M-1}x^{2m}$ is strictly increasing over $0<x<\infty$, there exists $c>0$ such that $|\alpha_n|^2=c$ for all $n\in\mathcal{P}$.   
Next, tightness gives that the rows have equal norm, implying that the first two rows have equal norm, i.e., $|\mathcal{P}|c=|\mathcal{P}|c^2$.  
Thus $c=1$, and so the nonzero bases are in the complex unit circle.  
Furthermore, since the zeroth and first rows have equal norm by tightness, we have $|\mathcal{P}|=N$, and so every base is in the complex unit circle.

Now consider the inner product between Vandermonde frame elements whose bases $\{e^{2\pi i x_n}\}_{n=1}^N$ come from the complex unit circle:
\[ \langle f_n,f_{n'}\rangle=\sum_{m=0}^{M-1}(e^{2\pi i x_n})^m\overline{(e^{2\pi i x_{n'}})^m}=\sum_{m=0}^{M-1}e^{2\pi i (x_n-x_{n'})m}. \]
We will show that the worst-case coherence comes from the two closest bases.  
Consider the following function:
\begin{equation}
\label{eq.g function}
g(x)
:=\bigg|\sum_{m=0}^{M-1}e^{2\pi ixm}\bigg|^2.
\end{equation}
Figure~\ref{figure} gives a plot of this function in the case where $M=5$.
We will prove two things about this function:
\begin{itemize}
\item[(i)] $\tfrac{d}{dx}g(x)<0$ for every $x\in(0,\tfrac{1}{2M})$,
\item[(ii)] $g(x)\leq g(\tfrac{1}{2M})$ for every $x\in(\tfrac{1}{2M},1-\tfrac{1}{2M})$.
\end{itemize}

\begin{figure}[t]
\centering
\includegraphics[width=0.5\textwidth]{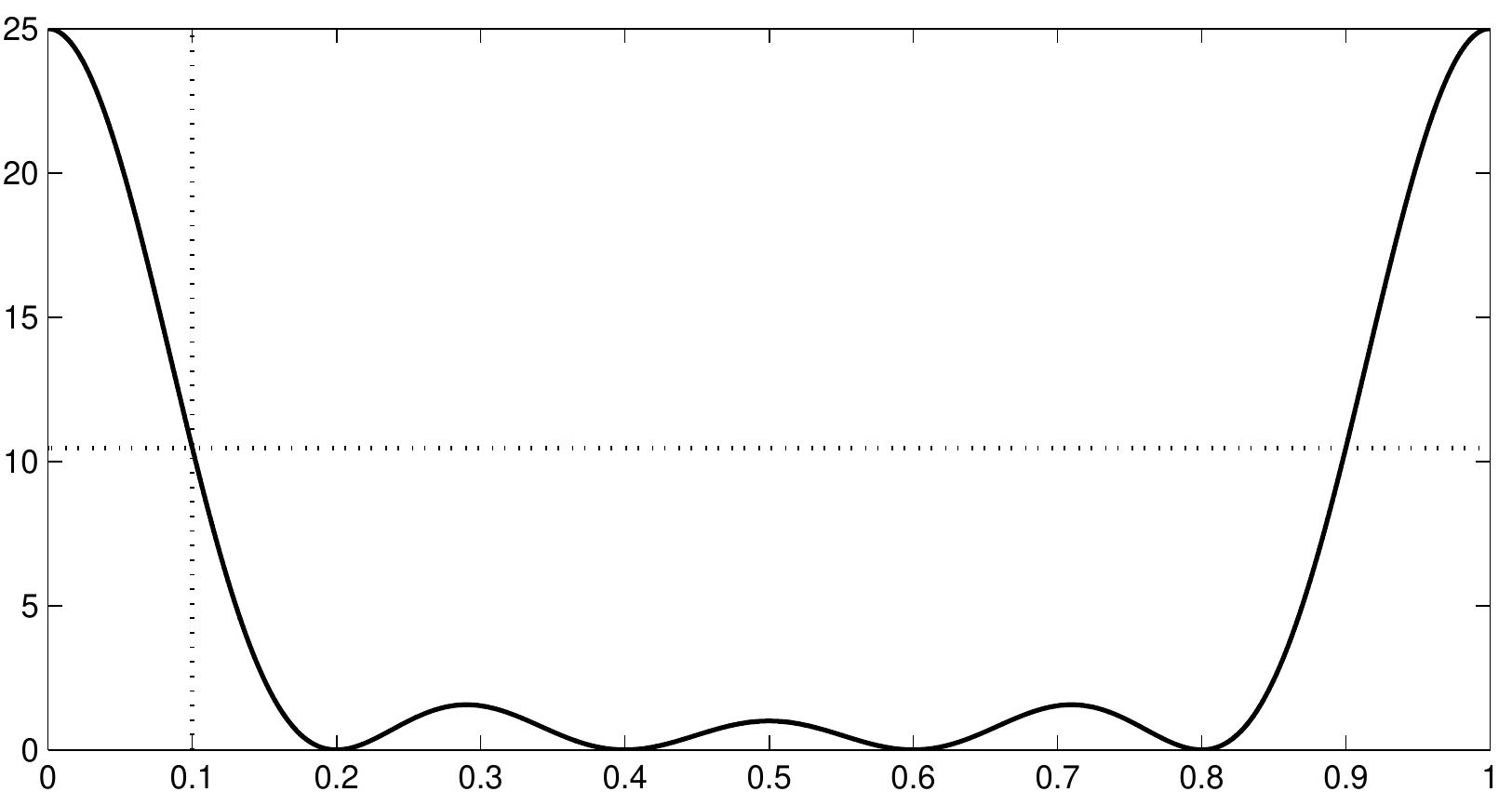}
\caption{Plot of $g$ defined by \eqref{eq.g function} in the case where $M=5$.  Observe (i) that $g$ is strictly decreasing on the interval $(0,\frac{1}{10})$, and (ii) that $g(x)\leq g(\frac{1}{10})$ for every $x\in(\frac{1}{10},\frac{9}{10})$.  As established in the proof of Theorem~\ref{thm.vandermonde}, $g$ behaves in this manner for general values of $M$.
\label{figure}}
\end{figure}

First, we claim that (i) and (ii) are sufficient to prove our result.
To establish this, we first show that the two closest bases $e^{2\pi ix_{n'}}$ and $e^{2\pi ix_{n''}}$ satisfy $|x_{n'}-x_{n''}|\leq\frac{1}{2M}$.
Without loss of generality, the $n$'s are ordered in such a way that $\{x_n\}_{n=0}^{N-1}\subseteq[0,1)$ are nondecreasing.
Define
\begin{equation*}
d(x_n,x_{n+1})
:=\left\{\begin{array}{ll}
x_{n+1}-x_n,&n=0,\ldots,N-2\\ 
x_0-(x_{N-1}-1),&n=N-1,
\end{array}
\right.
\end{equation*}
and let $n'$ be the $n$ which minimizes $d(x_n,x_{n+1})$.
Since the minimum is less than the average, we have
\begin{equation}
\label{eq.average inequality}
d(x_{n'},x_{n'+1})\leq\frac{1}{N}\bigg((x_0-(x_{N-1}-1))+\sum_{n=0}^{N-1}(x_{n+1}-x_n)\bigg)=\frac{1}{N}\leq\frac{1}{2M},
\end{equation}
provided $N\geq2M$.
Note that if we view $\{x_n\}_{n\in\mathbb{Z}_N}$ as members of $\mathbb{R}/\mathbb{Z}$, then $d(x_n,x_{n+1})=x_{n+1}-x_n$.
Since $g(x)$ is even, then (i) implies that $|\langle f_{n'+1},f_{n'}\rangle|^2=g(x_{n'+1}-x_{n'})$ is larger than any other $g(x_p-x_{p'})=|\langle f_p,f_{p'}\rangle|^2$ in which $x_p-x_{p'}\in[0,\tfrac{1}{2M}]\cup[1-\tfrac{1}{2M},1)$.
Next, \eqref{eq.average inequality} and (ii) together imply that $|\langle f_{n'+1},f_{n'}\rangle|^2=g(x_{n'+1}-x_{n'})\geq g(\tfrac{1}{2M})$ is larger than any other $g(x_p-x_{p'})=|\langle f_p,f_{p'}\rangle|^2$ in which $x_p-x_{p'}\in(\tfrac{1}{2M},1-\tfrac{1}{2M})$, provided $N\geq2M$.
Combined, (i) and (ii) give that $|\langle f_{n'+1},f_{n'}\rangle|$ achieves the worst-case coherence of $\{f_n\}_{n\in\mathbb{Z}_N}$.
Additionally, (i) gives that the worst-case coherence $|\langle f_{n'+1},f_{n'}\rangle|$ is minimized when $x_{n'+1}-x_{n'}$ is maximized, i.e., when the $x_n$'s are equally spaced in the unit interval.

To prove (i), note that the geometric sum formula gives
\begin{equation}
\label{eq.g of x}
g(x)
=\bigg|\sum_{m=0}^{M-1}e^{2\pi ixm}\bigg|^2
=\bigg|\frac{e^{2M\pi ix}-1}{e^{2\pi ix}-1}\bigg|^2
=\frac{2-2\cos(2M\pi x)}{2-2\cos(2\pi x)}
=\bigg(\frac{\sin(M\pi x)}{\sin(\pi x)}\bigg)^2,
\end{equation}
where the final expression uses the identity $1-\cos(2z)=2\sin^2 z$.
To show that $g$ is decreasing over $(0,\frac{1}{2M})$, note that the base of \eqref{eq.g of x} is positive on this interval, and performing the quotient rule to calculate its derivative will produce a fraction whose denominator is nonnegative and whose numerator is given by
\begin{equation}
\label{eq.g of x 2}
M\pi\sin(\pi x)\cos(M\pi x)-\pi\sin(M\pi x)\cos(\pi x).
\end{equation}
This factor is zero at $x=0$ and has derivative:
\begin{equation*}
-(M^2-1)\pi^2\sin(\pi x)\sin(M\pi x),
\end{equation*}
which is strictly negative for all $x\in(0,\tfrac{1}{2M})$.
Hence, \eqref{eq.g of x 2} is strictly negative whenever $x\in(0,\tfrac{1}{2M})$, and so $g'(x)<0$ for every $x\in(0,\tfrac{1}{2M})$.

For (ii), note that for every $x\in(\frac{1}{2M},1-\frac{1}{2M})$, we can individually bound the numerator and denominator of what the geometric sum formula gives:
\begin{equation*}
g(x)
=\bigg|\sum_{m=0}^{M-1}e^{2\pi ixm}\bigg|^2
=\frac{|e^{2M\pi ix}-1|^2}{|e^{2\pi ix}-1|^2}
\leq\frac{|e^{\pi i}-1|^2}{|e^{\pi i/M}-1|^2}
=\bigg|\sum_{m=0}^{M-1}e^{\pi im/M}\bigg|^2
=g(\tfrac{1}{2M}).\qedhere
\end{equation*}
\end{proof}

Consider the $N\times N$ discrete Fourier transform (DFT) matrix, scaled to have entries of unit modulus:
\begin{equation*}
\begin{bmatrix}
   1 & 1 & 1 & \cdots & 1\\
   1 & \omega & \omega^2 & \cdots & \omega^{N-1}\\
   1 & \omega^2 & \omega^4 & \cdots & \omega^{2(N-1)}\\
\vdots & \vdots & \vdots & \cdots & \vdots\\
   1 & \omega^{N-1} & \omega^{2(N-1)} & \cdots & \omega^{(N-1)(N-1)}
  \end{bmatrix},
\end{equation*}
where $\omega=e^{-2\pi i/N}$.
The first $M$ rows of the DFT form a Vandermonde matrix of distinct bases $\{\omega^n\}_{n=0}^{N-1}$; as such, this matrix is full spark by Lemma~\ref{lem.vandermonde}.
In fact, the previous result says that this is in some sense an optimal Vandermonde frame, but this might not be the best way to pick rows from a DFT.
Indeed, several choices of DFT rows could produce full spark frames, some with smaller coherence or other desirable properties, and so the remainder of this section focuses on full spark DFT submatrices.
First, we note that not every DFT submatrix is full spark.
For example, consider the $4\times 4$ DFT:
\begin{equation*}
\begin{bmatrix}
1&1&1&1\\
1&-i&-1&i\\
1&-1&1&-1\\
1&i&-1&-i
\end{bmatrix}.
\end{equation*}
Certainly, the zeroth and second rows of this matrix are not full spark, since the zeroth and second columns of this submatrix form the all-ones matrix, which is not invertible.
So what can be said about the set of permissible row choices?
The following result gives some necessary conditions on this set:

\begin{thm}
\label{thm.closure}
Take an $N\times N$ discrete Fourier transform matrix, and select the rows indexed by $\mathcal{M}\subseteq\mathbb{Z}_N$ to build the matrix $F$.
If $F$ is full spark, then so is the matrix built from rows indexed by
\begin{itemize}
\item[(i)] any translation of $\mathcal{M}$,
\item[(ii)] any $A\mathcal{M}$ with $A$ relatively prime to $N$,
\item[(iii)] the complement of $\mathcal{M}$ in $\mathbb{Z}_N$.
\end{itemize}
\end{thm}

\begin{proof}
For (i), we first define $D$ to be the $N\times N$ diagonal matrix whose diagonal entries are $\{\omega^n\}_{n=0}^{N-1}$.
Note that, since $\omega^{(m+1)n}=\omega^n\omega^{mn}$, translating the row indices $\mathcal{M}$ by $1$ corresponds to multiplying $F$ on the right by $D$.
For some set $\mathcal{K}\subseteq\mathbb{Z}_N$ of size $M:=|\mathcal{M}|$, let $F_\mathcal{K}$ denote the $M\times M$ submatrix of $F$ whose columns are indexed by $\mathcal{K}$, and let $D_\mathcal{K}$ denote the $M\times M$ diagonal submatrix of $D$ whose diagonal entries are indexed by $\mathcal{K}$.
Then since $D_\mathcal{K}$ is unitary, we have
\begin{equation*}
|\mathrm{det}((FD)_\mathcal{K})|
=|\mathrm{det}(F_\mathcal{K}D_\mathcal{K})|
=|\mathrm{det}(F_\mathcal{K})||\mathrm{det}(D_\mathcal{K})|
=|\mathrm{det}(F_\mathcal{K})|.
\end{equation*}
Thus, if $F$ is full spark, $|\mathrm{det}((FD)_\mathcal{K})|=|\mathrm{det}(F_\mathcal{K})|>0$, and so $FD$ is also full spark.
Using this fact inductively proves (i) for all translations of $\mathcal{M}$.

For (ii), let $G$ denote the submatrix of rows indexed by $A\mathcal{M}$.
Then for any set $\mathcal{K}\subseteq\mathbb{Z}_N$ of size $M$,
\begin{equation*}
\mathrm{det}(G_\mathcal{K})
=\mathrm{det}(\omega^{(Am)k})_{m\in\mathcal{M},k\in\mathcal{K}}
=\mathrm{det}(\omega^{m(Ak)})_{m\in\mathcal{M},k\in\mathcal{K}}
=\mathrm{det}(F_{A\mathcal{K}}).
\end{equation*}
Since $A$ is relatively prime to $N$, multiplication by $A$ permutes the elements of $\mathbb{Z}_N$, and so $A\mathcal{K}$ has exactly $M$ distinct elements.
Thus, if $F$ is full spark, then $\mathrm{det}(G_\mathcal{K})=\mathrm{det}(F_{A\mathcal{K}})\neq0$, and so $G$ is also full spark.

For (iii), we let $G$ be the $(N-M)\times N$ submatrix of rows indexed by $\mathcal{M}^\mathrm{c}$, so that
\begin{equation}
\label{eq.complements}
NI_N
=\begin{bmatrix}F^* & G^*\end{bmatrix}\begin{bmatrix}F\\G\end{bmatrix}
=F^*F+G^*G.
\end{equation}
We will use contraposition to show that $F$ being full spark implies that $G$ is also full spark.
To this end, suppose $G$ is not full spark.
Then $G$ has a collection of $N-M$ linearly dependent columns $\{g_i\}_{i\in \mathcal{K}}$, and so there exists a nontrivial sequence $\{\alpha_i\}_{i\in \mathcal{K}}$ such that 
\begin{equation*}
\sum_{i\in \mathcal{K}}\alpha_ig_i
=0.
\end{equation*}
Considering $g_i=G\delta_i$, where $\delta_i$ is the $i$th identity basis element, we can use \eqref{eq.complements} to express this linear dependence in terms of $F$:
\begin{equation*}
0
=G^*0
=G^*\sum_{i\in \mathcal{K}}\alpha_ig_i
=\sum_{i\in \mathcal{K}}\alpha_iG^*G\delta_i
=\sum_{i\in \mathcal{K}}\alpha_i(NI_N-F^*F)\delta_i.
\end{equation*}
Rearranging then gives
\begin{equation}
\label{eq.x defn}
x:=N\sum_{i\in \mathcal{K}}\alpha_i\delta_i=\sum_{i\in \mathcal{K}}\alpha_iF^*F\delta_i.
\end{equation}
Here, we note that $x$ is nonzero since $\{\alpha_i\}_{i\in \mathcal{K}}$ is nontrivial, and that $x\in\mathrm{Range}(F^*F)$.
Furthermore, whenever $j\not\in \mathcal{K}$, we have from \eqref{eq.x defn} that
\begin{equation*}
\langle x,F^*F\delta_j\rangle
=\langle F^*Fx,\delta_j\rangle
=N\bigg\langle F^*F\sum_{i\in \mathcal{K}}\alpha_i\delta_i,\delta_j\bigg\rangle
=N^2\bigg\langle\sum_{i\in \mathcal{K}}\alpha_i\delta_i,\delta_j\bigg\rangle
=0,
\end{equation*}
and so $x\perp\mathrm{Span}\{F^*F\delta_j\}_{j\in \mathcal{K}^\mathrm{c}}$.
Thus, the containment $\mathrm{Span}\{F^*F\delta_j\}_{j\in \mathcal{K}^\mathrm{c}}\subseteq\mathrm{Range}(F^*F)$ is proper, and so
\begin{equation*}
M
=\mathrm{Rank}(F)
=\mathrm{Rank}(F^*F)
>\mathrm{Rank}(F^*F_{\mathcal{K}^\mathrm{c}})
=\mathrm{Rank}(F_{\mathcal{K}^\mathrm{c}}).
\end{equation*}
Since the $M\times M$ submatrix $F_{\mathcal{K}^\mathrm{c}}$ is rank-deficient, it is not invertible, and therefore $F$ is not full spark.
\end{proof}

We note that our proof of (iii) above uses techniques from Cahill et al.~\cite{CahillCH:sampta11}, and can be easily generalized to prove that the Naimark complement of a full spark tight frame is also full spark.
Theorem~\ref{thm.closure} tells us quite a bit about the set of permissible choices for DFT rows.
For example, not only can we pick the first $M$ rows of the DFT to produce a full spark Vandermonde frame, but we can also pick any consecutive $M$ rows, by Theorem~\ref{thm.closure}(i).
We would like to completely characterize the choices that produce full spark harmonic frames.
The following classical result does this in the case where $N$ is prime:

\begin{thm}[Chebotar\"{e}v, see~\cite{StevenhagenL:mi96}]
Let $N$ be prime.
Then every square submatrix of the $N\times N$ discrete Fourier transform matrix is invertible.
\end{thm}

As an immediate consequence of Chebotar\"{e}v's theorem, every choice of rows from the DFT produces a full spark harmonic frame, provided $N$ is prime.
This application of Chebotar\"{e}v's theorem was first used by Cand\`{e}s et al.~\cite{CandesRT:06} for sparse signal processing.
Note that each of these frames are equal-norm and tight by construction.
Harmonic frames can also be designed to have minimal coherence; Xia et al.~\cite{XiaZG:it05} produces harmonic equiangular tight frames by selecting row indices which form a difference set in $\mathbb{Z}_N$.
Interestingly, most known families of difference sets in $\mathbb{Z}_N$ require $N$ to be prime~\cite{JungnickelPS:hcd07}, and so the corresponding harmonic equiangular tight frames are guaranteed to be full spark by Chebotar\"{e}v's theorem.
In the following, we use Chebotar\"{e}v's theorem to demonstrate full spark for a class of frames which contains harmonic frames, namely, frames which arise from concatenating harmonic frames with any number of identity basis elements:

\begin{thm}[{cf.~\cite[Theorem 1.1]{Tao:mrl05}}]
\label{thm.harmonic plus identity}
Let $N$ be prime, and pick any $M\leq N$ rows of the $N\times N$ discrete Fourier transform matrix to form the harmonic frame $H$.
Next, pick any $K\leq M$, and take $D$ to be the $M\times M$ diagonal matrix whose first $K$ diagonal entries are $\sqrt{\frac{N+K-M}{MN}}$, and whose remaining $M-K$ entries are $\sqrt{\frac{N+K}{MN}}$.
Then concatenating $DH$ with the first $K$ identity basis elements produces an $M\times(N+K)$ full spark unit norm tight frame.
\end{thm}

As an example, when $N=5$ and $K=1$, we can pick $M=3$ rows of the $5\times 5$ DFT which are indexed by $\{0,1,4\}$.  
In this case, $D$ makes the entries of the first DFT row have size $\sqrt{\frac{1}{5}}$ and the entries of the remaining rows have size $\sqrt{\frac{2}{5}}$.
Concatenating with the first identity basis element then produces an equiangular tight frame which is full spark:
\begin{equation}
\label{eq.full spark etf example}
F
=\left[\begin{array}{llllll} 
\sqrt{\frac{1}{5}}\quad\quad\quad&\sqrt{\frac{1}{5}}\quad\quad\quad&\sqrt{\frac{1}{5}}&\sqrt{\frac{1}{5}}\quad\quad\quad&\sqrt{\frac{1}{5}}\quad\quad\quad&1\\
\sqrt{\frac{2}{5}}&\sqrt{\frac{2}{5}}e^{-2\pi\mathrm{i}/5}&\sqrt{\frac{2}{5}}e^{-2\pi\mathrm{i}2/5}&\sqrt{\frac{2}{5}}e^{-2\pi\mathrm{i}3/5}&\sqrt{\frac{2}{5}}e^{-2\pi\mathrm{i}4/5}&0\\
\sqrt{\frac{2}{5}}&\sqrt{\frac{2}{5}}e^{-2\pi\mathrm{i}4/5}&\sqrt{\frac{2}{5}}e^{-2\pi\mathrm{i}3/5}&\sqrt{\frac{2}{5}}e^{-2\pi\mathrm{i}2/5}&\sqrt{\frac{2}{5}}e^{-2\pi\mathrm{i}/5}&0\\     
       \end{array}\right].
\end{equation}

\begin{proof}[Proof of Theorem~\ref{thm.harmonic plus identity}]
Let $F$ denote the resulting $M\times(N+K)$ frame.
We start by verifying that $F$ is unit norm.
Certainly, the identity basis elements have unit norm.
For the remaining frame elements, the modulus of each entry is determined by $D$, and so the norm squared of each frame element is
\begin{equation*}
K(\tfrac{N+K-M}{MN})+(M-K)(\tfrac{N+K}{MN})
=1.
\end{equation*}
To demonstrate that $F$ is tight, it suffices to show that $FF^*=\frac{N+K}{M}I_M$.
The rows of $DH$ are orthogonal since they are scaled rows of the DFT, while the rows of the identity portion are orthogonal because they have disjoint support.
Thus, $FF^*$ is diagonal.
Moreover, the norm squared of each of the first $K$ rows is $N(\frac{N+K-M}{MN})+1=\frac{N+K}{M}$, while the norm squared of each of the remaining rows is $N(\frac{N+K}{MN})=\frac{N+K}{M}$, and so $FF^*=\frac{N+K}{M}I_M$.

To demonstrate that $F$ is full spark, first note that every $M\times M$ submatrix of $DH$ is invertible since
\begin{equation*}
|\mathrm{det}((DH)_\mathcal{K})|
=|\mathrm{det}(DH_\mathcal{K})|
=|\mathrm{det}(D)||\mathrm{det}(H_\mathcal{K})|
>0,
\end{equation*}
by Chebotar\"{e}v's theorem.
Also, in the case where $K=M$, we note that the $M\times M$ submatrix of $F$ composed solely of identity basis elements is trivially invertible.
The only remaining case to check is when identity basis elements and columns of $DH$ appear in the same $M\times M$ submatrix $F_\mathcal{K}$.
In this case, we may shuffle the rows of $F_\mathcal{K}$ to have the form 
\begin{equation*}
\begin{bmatrix} A&0\\B&I_K\end{bmatrix}.
\end{equation*}
Since shuffling rows has no impact on the size of the determinant, we may further use a determinant identity on block matrices to get
\begin{equation*}
|\mathrm{det}(F_\mathcal{K})|
=\bigg|\mathrm{det}\begin{bmatrix} A&0\\B&I_K\end{bmatrix}\bigg|
=|\mathrm{det}(A)\mathrm{det}(I_K)|
=|\mathrm{det}(A)|.
\end{equation*}
Since $A$ is a multiple of a square submatrix of the $N\times N$ DFT, we are done by Chebotar\"{e}v's theorem.
\end{proof}

As an example of Theorem~\ref{thm.harmonic plus identity}, pick $N$ to be a prime congruent to $1\bmod 4$, and select $\frac{N+1}{2}$ rows of the $N\times N$ DFT according to the index set $\mathcal{M}:=\{k^2:k\in\mathbb{Z}_N\}$.
If we take $K=1$, the process in Theorem~\ref{thm.harmonic plus identity} produces an equiangular tight frame of redundancy $2$, which can be verified using quadratic Gauss sums; in the case where $N=5$, this construction produces \eqref{eq.full spark etf example}.
Note that this corresponds to a special case of a construction in Zauner's thesis~\cite{Zauner:phd99}, which was later studied by Renes~\cite{Renes:laa07} and Strohmer~\cite{Strohmer:laa08}.
Theorem~\ref{thm.harmonic plus identity} says that this construction is full spark.

Maximally sparse frames have recently become a subject of active research~\cite{CasazzaHKK:it,FickusMT:laa11}.
We note that when $K=M$, Theorem~\ref{thm.harmonic plus identity} produces a maximally sparse $M\times (N+K)$ full spark frame, having a total of $M(M-1)$ zero entries.
To see that this sparsity level is maximal, we note that if the frame had any more zero entries, then at least one of the rows would have $M$ zero entries, meaning the corresponding $M\times M$ submatrix would have a row of all zeros and hence a zero determinant.
Similar ideas were studied previously by Nakamura and Masson~\cite{NakamuraM:ieeetc82}.

Another interesting case is where $K=M=N$, i.e., when the frame constructed in Theorem~\ref{thm.harmonic plus identity} is a union of the unitary DFT and identity bases.
Unions of orthonormal bases have received considerable attention in the context of sparse approximation~\cite{DonohoE:pnas03,Tropp:acha08}.
In fact, when $N$ is a perfect square, concatenating the DFT with an identity basis forms the canonical example $F$ of a dictionary with small spark~\cite{DonohoE:pnas03}.
To be clear, the Dirac comb of $\sqrt{N}$ spikes is an eigenvector of the DFT, and so concatenating this comb with the negative of its Fourier transform produces a $2\sqrt{N}$-sparse vector in the nullspace of $F$.
In stark contrast, when $N$ is prime, Theorem~\ref{thm.harmonic plus identity} shows that $F$ is full spark.

The vast implications of Chebotar\"{e}v's theorem leads one to wonder whether the result admits any interesting generalization.
In this direction, Cand\`{e}s et al.~\cite{CandesRT:06} note that any such generalization must somehow account for the nontrivial subgroups of $\mathbb{Z}_N$ which are not present when $N$ is prime.
Certainly, if one could characterize the full spark submatrices of a general DFT, this would provide ample freedom to optimize full spark frames for additional considerations.
While we do not have a characterization for the general case, we do have one for the case where $N$ is a prime power.
Before stating the result, we require a definition:

\begin{defn}
We say a subset $\mathcal{M}\subseteq\mathbb{Z}_N$ is \emph{uniformly distributed over the divisors of $N$} if, for every divisor $d$ of $N$, the $d$ cosets of $\langle d\rangle$ partition $\mathcal{M}$ into subsets, each of size $\lfloor\frac{|\mathcal{M}|}{d}\rfloor$ or $\lceil\frac{|\mathcal{M}|}{d}\rceil$.
\end{defn}

At first glance, this definition may seem rather unnatural, but we will discover some important properties of uniformly distributed rows from the DFT.
As an example, we take a short detour by considering the \emph{restricted isometry property (RIP)}, which has received considerable attention recently for its use in compressed sensing.
We say a matrix $F$ is $(K,\delta)$-RIP if
\begin{equation*}
(1-\delta)\|x\|^2\leq\|Fx\|^2\leq(1+\delta)\|x\|^2\qquad\mbox{whenever }\|x\|_0\leq K.
\end{equation*}
Cand\`{e}s and Tao~\cite{CandesT:it05} demonstrated that the sparsest $\Psi^*x$ which satisfies \eqref{eq.sparse signal processing} can be found using $\ell_1$-minimization, provided $F$ is $(2K,\sqrt{2}-1)$-RIP.
Later, Rudelson and Vershynin~\cite{RudelsonV:cpaa08} showed that a matrix of random rows from a DFT and normalized columns is RIP with high probability. 
We will show that harmonic frames satisfy RIP only if the selected row indices are nearly uniformly distributed over sufficiently small divisors of $N$.

To this end, recall that for any divisor $d$ of $N$, the Fourier transform of the $d$-sparse normalized Dirac comb $\frac{1}{\sqrt{d}}\chi_{\langle \frac{N}{d}\rangle}$ is the $\frac{N}{d}$-sparse normalized Dirac comb $\sqrt{\frac{d}{N}}\chi_{\langle d\rangle}$.
Let $F$ be the $N\times N$ unitary DFT, and let $H$ be the harmonic frame which arises from selecting rows of $F$ indexed by $\mathcal{M}$ and then normalizing the columns.
In order for $H$ to be $(K,\delta)$-RIP, $\mathcal{M}$ must contain at least one member of $\langle d\rangle$ for every divisor $d$ of $N$ which is $\leq K$, since otherwise
\begin{equation*}
H\tfrac{1}{\sqrt{d}}\chi_{\langle\frac{N}{d}\rangle}
=\sqrt{\tfrac{N}{|\mathcal{M}|}}(F\tfrac{1}{\sqrt{d}}\chi_{\langle\frac{N}{d}\rangle})_\mathcal{M}
=\sqrt{\tfrac{N}{|\mathcal{M}|}}\Big(\sqrt{\tfrac{d}{N}}\chi_{\langle d\rangle}\Big)_\mathcal{M}
=\sqrt{\tfrac{d}{|\mathcal{M}|}}\chi_{\mathcal{M}\cap\langle d\rangle}
=0, 
\end{equation*}
which violates the lower RIP bound at $x=\frac{1}{\sqrt{d}}\chi_{\langle\frac{N}{d}\rangle}$.
In fact, the RIP bounds indicate that 
\begin{equation*}
\|Hx\|^2
=\|H\tfrac{1}{\sqrt{d}}\chi_{\langle\frac{N}{d}\rangle}\|^2
=\Big\|\sqrt{\tfrac{d}{|\mathcal{M}|}}\chi_{\mathcal{M}\cap\langle d\rangle}\Big\|^2
=\tfrac{d}{|\mathcal{M}|}|\mathcal{M}\cap\langle d\rangle|
\end{equation*}
cannot be more than $\delta$ away from $\|x\|^2=1$.
Similarly, taking $x$ to be $\frac{1}{\sqrt{d}}\chi_{\langle\frac{N}{d}\rangle}$ modulated by $a$, i.e., $x[n]:=\frac{1}{\sqrt{d}}\chi_{\langle\frac{N}{d}\rangle}[n]e^{2\pi ian/N}$ for every $n\in\mathbb{Z}_N$, gives that $\|Hx\|^2=\frac{d}{|\mathcal{M}|}|\mathcal{M}\cap(a+\langle d\rangle)|$ is also no more than $\delta$ away from $1$.
This observation gives the following result:

\begin{thm}
Select rows indexed by $\mathcal{M}\subseteq\mathbb{Z}_N$ from the $N\times N$ discrete Fourier transform matrix and then normalize the columns to produce the harmonic frame $H$.
Then $H$ satisfies the $(K,\delta)$-restricted isometry property only if
\begin{equation*}
\Big|\big|\mathcal{M}\cap(a+\langle d\rangle)\big|-\tfrac{|\mathcal{M}|}{d}\Big|\leq\tfrac{|\mathcal{M}|}{d}\delta
\end{equation*}
for every divisor $d$ of $N$ with $d\leq K$ and every $a=0,\ldots,d-1$.
\end{thm}

Now that we have an intuition for uniform distribution in terms of modulated Dirac combs and RIP, we take this condition to the extreme by considering uniform distribution over all divisors.
Doing so produces a complete characterization of full spark harmonic frames when $N$ is a prime power:

\begin{thm}
\label{thm.uniformly distributed}
Let $N$ be a prime power, and select rows indexed by $\mathcal{M}\subseteq\mathbb{Z}_N$ from the $N\times N$ discrete Fourier transform matrix to build the submatrix $F$.
Then $F$ is full spark if and only if $\mathcal{M}$ is uniformly distributed over the divisors of $N$.
\end{thm}

Note that, perhaps surprisingly, an index set $\mathcal{M}$ can be uniformly distributed over $p$ but not over $p^2$, and vice versa.
For example, $\mathcal{M}=\{0,1,4\}$ is uniformly distributed over $2$ but not $4$, while $\mathcal{M}=\{0,2\}$ is uniformly distributed over $4$ but not $2$.

Since the first $M$ rows of a DFT form a full spark Vandermonde matrix, let's check that this index set is uniformly distributed over the divisors of $N$.
For each divisor $d$ of $N$, we partition the first $M$ indices into the $d$ cosets of $\langle d\rangle$.
Write $M=qd+r$ with $0\leq r<d$.
The first $qd$ of the $M$ indices are distributed equally amongst all $d$ cosets, and then the remaining $r$ indices are distributed equally amongst the first $r$ cosets.
Overall, the first $r$ cosets contain $q+1=\lfloor\frac{M}{d}\rfloor+1$ indices, while the remaining $d-r$ cosets have $q=\lfloor\frac{M}{d}\rfloor$ indices; thus, the first $M$ indices are indeed uniformly distributed over the divisors of $N$.
Also, when $N$ is prime, \emph{every} subset of $\mathbb{Z}_N$ is uniformly distributed over the divisors of $N$ in a trivial sense.
In fact, Chebotar\"{e}v's theorem follows immediately from Theorem~\ref{thm.uniformly distributed}.
In some ways, portions of our proof of Theorem~\ref{thm.uniformly distributed} mirror recurring ideas in the existing proofs of Chebotar\"{e}v's theorem~\cite{DelvauxV:laa08,EvansI:ams77,StevenhagenL:mi96,Tao:mrl05}.
For the sake of completeness, we provide the full argument and save the reader from having to parse portions of proofs from multiple references.
We start with the following lemmas, whose proofs are based on the proofs of Lemmas 1.2 and 1.3 in~\cite{Tao:mrl05}.

\begin{lem}
\label{lem.multiple of p}
Let $N$ be a power of some prime $p$, and let $P(z_1,\ldots,z_M)$ be a polynomial with integer coefficients.
Suppose there exists $N$th roots of unity $\{\omega_m\}_{m=1}^M$ such that $P(\omega_1,\ldots,\omega_M)=0$.
Then $P(1,\ldots,1)$ is a multiple of $p$.
\end{lem}

\begin{proof}
Denoting $\omega:=e^{-2\pi i/N}$, then for every $m=1,\ldots,M$, we have $\omega_m=\omega^{k_m}$ for some $0\leq k_m<N$.
Defining the polynomial $Q(z):=P(z^{k_1},\ldots,z^{k_M})$, we have $Q(\omega)=0$ by assumption.
Also, $Q(z)$ is a polynomial with integer coefficients, and so it must be divisible by the minimal polynomial of $\omega$, namely, the cyclotomic polynomial $\Phi_N(z)$.
Evaluating both polynomials at $z=1$ then gives that $p=\Phi_N(1)$ divides $Q(1)=P(1,\ldots,1)$.
\end{proof}

\begin{lem}
\label{lem.big fraction}
Let $N$ be a power of some prime $p$, and pick $\mathcal{M}=\{m_i\}_{i=1}^M\subseteq\mathbb{Z}_N$ such that
\begin{equation}
\label{eq.big fraction}
\frac{\displaystyle{\prod_{1\leq i<j\leq M}(m_j-m_i)}}{\displaystyle{\prod_{m=0}^{M-1}m!}}
\end{equation}
is not a multiple of $p$.
Then the rows indexed by $\mathcal{M}$ in the $N\times N$ discrete Fourier transform form a full spark frame. 
\end{lem}

\begin{proof}
We wish to show that $\mathrm{det}(\omega_n^m)_{m\in\mathcal{M},1\leq n\leq M}\neq0$ for all $M$-tuples of distinct $N$th roots of unity $\{\omega_n\}_{n=1}^M$.
Define the polynomial $D(z_1,\ldots,z_M):=\mathrm{det}(z_n^m)_{m\in\mathcal{M},1\leq n\leq M}$.
Since columns $i$ and $j$ of $(z_n^m)_{m\in\mathcal{M},1\leq n\leq M}$ are identical whenever $z_i=z_j$, we know that $D$ vanishes in each of these instances, and so we can factor:
\begin{equation*}
D(z_1,\ldots,z_M)=P(z_1,\ldots,z_M)\prod_{1\leq i<j\leq M}(z_j-z_i)
\end{equation*}
for some polynomial $P(z_1,\ldots,z_M)$ with integer coefficients.
By Lemma~\ref{lem.multiple of p}, it suffices to show that $P(1,\ldots,1)$ is not a multiple of $p$, since this implies $D(\omega_1,\ldots,\omega_M)$ is nonzero for all $M$-tuples of distinct $N$th roots of unity $\{\omega_n\}_{n=1}^M$.

To this end, we proceed by considering
\begin{equation}
\label{eq.A defn}
A:=\bigg(z_1\frac{\partial}{\partial z_1}\bigg)^0\bigg(z_2\frac{\partial}{\partial z_2}\bigg)^1\cdots\bigg(z_M\frac{\partial}{\partial z_M}\bigg)^{M-1}D(z_1,\ldots,z_M)\bigg|_{z_1=\cdots=z_M=1}.
\end{equation}
To compute $A$, we note that each application of $z_j\frac{\partial}{\partial z_j}$ produces terms according to the product rule.
For some terms, a linear factor of the form $z_j-z_i$ or $z_i-z_j$ is replaced by $z_j$ or $-z_j$, respectively.
For each the other terms, these linear factors are untouched, while another factor, such as $P(z_1,\ldots,z_M)$, is differentiated and multiplied by $z_j$.
Note that there are a total of $M(M-1)/2$ linear factors, and only $M(M-1)/2$ differentiation operators to apply.
Thus, after expanding every product rule, there will be two types of terms: terms in which every differentiation operator was applied to a linear factor, and terms which have at least one linear factor remaining untouched.
When we evaluate at $z_1=\cdots=z_M=1$, the terms with linear factors vanish, and so the only terms which remain came from applying every differentiation operator to a linear factor.
Furthermore, each of these terms before the evaluation is of the form $P(z_1,\ldots,z_M)\prod_{1\leq i<j\leq M}z_j$, and so evaluation at $z_1=\cdots=z_M=1$ produces a sum of terms of the form $P(1,\ldots,1)$; to determine the value of $A$, it remains to count these terms.
The $M-1$ copies of $z_M\frac{\partial}{\partial z_M}$ can only be applied to linear factors of the form $z_M-z_i$, of which there are $M-1$, and so there are a total of $(M-1)!$ ways to distribute these operators.
Similarly, there are $(M-2)!$ ways to distribute the $M-2$ copies of $z_{M-1}\frac{\partial}{\partial z_{M-1}}$ amongst the $M-2$ linear factors of the form $z_{M-1}-z_i$.
Continuing in this manner produces an expression for $A$:
\begin{equation}
\label{eq.first A}
A=(M-1)!(M-2)!\cdots 1!0!~P(1,\ldots,1).
\end{equation}

For an alternate expression of $A$, we substitute the definition of $D(z_1,\ldots,z_M)$ into $\eqref{eq.A defn}$.
Here, we exploit the multilinearity of the determinant and the fact that $(z_n\frac{\partial}{\partial z_n})z_n^m=mz_n^m$ to get
\begin{equation}
\label{eq.second A}
A=\mathrm{det}(m^{n-1})_{m\in\mathcal{M},1\leq n\leq M}=\prod_{1\leq i<j\leq M}(m_j-m_i),
\end{equation}
where the final equality uses the fact that $(m^{n-1})_{m\in\mathcal{M},1\leq n\leq M}$ is the transpose of a Vandermonde matrix.
Equating \eqref{eq.first A} to \eqref{eq.second A} reveals that \eqref{eq.big fraction} is an expression for $P(1,\ldots,1)$.  
Thus, by assumption, $P(1,\ldots,1)$ is not a multiple of $p$, and so we are done.
\end{proof}

\begin{proof}[Proof of Theorem~\ref{thm.uniformly distributed}]
($\Leftarrow$) We will use Lemma~\ref{lem.big fraction} to demonstrate that $F$ is full spark.
To apply this lemma, we need to establish that \eqref{eq.big fraction} is not a multiple of $p$, and to do this, we will show that there are as many $p$-divisors in the numerator of \eqref{eq.big fraction} as there are in the denominator.
We start by counting the $p$-divisors of the denominator:
\begin{equation}
\label{eq.denominator 1}
\prod_{m=0}^{M-1}m!=\prod_{m=1}^{M-1}\prod_{\ell=1}^m \ell=\prod_{\ell=1}^{M-1}\prod_{m=1}^{M-l} \ell.
\end{equation}
For each pair of integers $k,a\geq 1$, there are $\max\{M-ap^k,~0\}$ factors in \eqref{eq.denominator 1} of the form $\ell=ap^k$.
By adding these, we count each factor $\ell$ as many times as it can be expressed as a multiple of a power of $p$, which equals the number of $p$-divisors in $\ell$.
Thus, the number of $p$-divisors of \eqref{eq.denominator 1} is
\begin{equation}
\label{eq.denominator 2}
\sum_{k=1}^{\lfloor \log_p M\rfloor}\sum_{a=1}^{\lfloor \frac{M}{p^k}\rfloor}(M-ap^k).
\end{equation}

Next, we count the $p$-divisors of the numerator of \eqref{eq.big fraction}.
To do this, we use the fact that $\mathcal{M}$ is uniformly distributed over the divisors of $N$. 
Since $N$ is a power of $p$, the only divisors of $N$ are smaller powers of $p$.
Also, the cosets of $\langle p^k\rangle$ partition $\mathcal{M}$ into subsets $S_{k,b}:=\{m_i\equiv b\mod p^k\}$.
We note that $m_j-m_i$ is a multiple of $p^k$ precisely when $m_i$ and $m_j$ belong to the same subset $S_{k,b}$ for some $0\leq b<p^k$.
To count $p$-divisors, we again count each factor $m_j-m_i$ as many times as it can be expressed as a multiple of a prime power:
\begin{equation}
\label{eq.numerator 1}
\sum_{k=1}^{\lfloor \log_p M\rfloor}\sum_{b=0}^{p^k-1}\binom{|S_{k,b}|}{2}.
\end{equation}
Write $M=qp^k+r$ with $0\leq r<p^k$.
Then $q=\lfloor\frac{M}{p^k}\rfloor$.
Since $\mathcal{M}$ is uniformly distributed over $p^k$, there are $r$ subsets $S_{k,b}$ with $q+1$ elements and $p^k-r$ subsets with $q$ elements.
We use this to get
\begin{equation*}
\sum_{b=0}^{p^k-1}\binom{|S_{k,b}|}{2}
=\binom{q+1}{2}r+\binom{q}{2}(p^k-r)
=\frac{q}{2}\Big((q-1)p^k+2r+(qp^k-qp^k)\Big).
\end{equation*}
Rearranging and substituting $M=qp^k+r$ then gives
\begin{equation*}
\sum_{b=0}^{p^k-1}\binom{|S_{k,b}|}{2}
=\frac{q}{2}\Big(2M-(q+1)p^k\Big)
=Mq-\binom{q+1}{2}p^k
=\sum_{a=1}^{\lfloor \frac{M}{p^k}\rfloor}(M-ap^k).
\end{equation*}
Thus, there are as many $p$-divisors in the numerator \eqref{eq.numerator 1} as there are in the denominator \eqref{eq.denominator 2}, and so \eqref{eq.big fraction} is not divisible by $p$.
Lemma~\ref{lem.big fraction} therefore gives that $F$ is full spark.

($\Rightarrow$)
We will prove that this direction holds regardless of whether $N$ is a prime power.
Suppose $\mathcal{M}\subseteq\mathbb{Z}_N$ is not uniformly distributed over the divisors of $N$.
Then there exists a divisor $d$ of $N$ such that one of the cosets of $\langle d\rangle$ intersects $\mathcal{M}$ with $\leq\lfloor \frac{M}{d}\rfloor-1$ or $\geq\lceil\frac{M}{d}\rceil+1$ indices.
Notice that if a coset of $\langle d\rangle$ intersects $\mathcal{M}$ with $\leq\lfloor \frac{M}{d}\rfloor-1$ indices, then the complement $\mathcal{M}^\mathrm{c}$ intersects the same coset with $\geq\lceil\frac{N-M}{d}\rceil+1=\lceil\frac{|\mathcal{M}^\mathrm{c}|}{d}\rceil+1$ indices.
By Theorem~\ref{thm.closure}(iii), $\mathcal{M}$ produces a full spark harmonic frame precisely when $\mathcal{M}^\mathrm{c}$ produces a full spark harmonic frame, and so we may assume without loss of generality that there exists a coset of $\langle d\rangle$ which intersects $\mathcal{M}$ with $\geq\lceil\frac{M}{d}\rceil+1$ indices.

To prove that the rows with indices in $\mathcal{M}$ are not full spark, we find column entries which produce a singular submatrix.
Writing $M=qd+r$ with $0\leq r<d$, let $\mathcal{K}$ contain $q=\lfloor\frac{M}{d}\rfloor$ cosets of $\langle\frac{N}{d}\rangle$ along with $r$ elements from an additional coset.
We claim that the DFT submatrix with row entries $\mathcal{M}$ and column entries $\mathcal{K}$ is singular.
To see this, shuffle the rows and columns to form a matrix $A$ in which the row entries are grouped into common cosets of $\langle d\rangle$ and the column entries are grouped into common cosets of $\langle\frac{N}{d}\rangle$.
This breaks $A$ into rank-1 submatrices: each pair of cosets $a+\langle d\rangle$ and $b+\langle\frac{N}{d}\rangle$ produces a submatrix
\begin{equation*}
(\omega^{(a+id)(b+j\frac{N}{d})})_{i\in\mathcal{I},j\in\mathcal{J}}
=\omega^{ab}(\omega^{bdi}\omega^{a\frac{N}{d}j})_{i\in\mathcal{I},j\in\mathcal{J}}
\end{equation*}
for some index sets $\mathcal{I}$ and $\mathcal{J}$; this is a rank-1 outer product.
Let $\mathcal{L}$ be the largest intersection between $\mathcal{M}$ and a coset of $\langle d\rangle$.
Then $|\mathcal{L}|\geq\lceil\frac{M}{d}\rceil+1$ is the number of rows in the tallest of these rank-1 submatrices.
Define $A_\mathcal{L}$ to be the $M\times M$ matrix with entries $A_\mathcal{L}[i,j]=A[i,j]$ whenever $i\in\mathcal{L}$ and zero otherwise.
Then 
\begin{equation}
\label{eq.rank 1}
\mathrm{Rank}(A)
=\mathrm{Rank}(A_\mathcal{L}+A-A_\mathcal{L})
\leq\mathrm{Rank}(A_\mathcal{L})+\mathrm{Rank}(A-A_\mathcal{L}).
\end{equation}
Since $A-A_\mathcal{L}$ has $|\mathcal{L}|$ rows of zero entries, we also have
\begin{equation}
\label{eq.rank 2}
\mathrm{Rank}(A-A_\mathcal{L})
\leq M-|\mathcal{L}|\leq M-(\lceil\tfrac{M}{d}\rceil+1).
\end{equation}
Moreover, since we can decompose $A_\mathcal{L}$ into a sum of $\lceil\frac{M}{d}\rceil$ zero-padded rank-1 submatrices, we have $\mathrm{Rank}(A_\mathcal{L})
\leq \lceil\frac{M}{d}\rceil$.
Combining this with \eqref{eq.rank 1} and \eqref{eq.rank 2} then gives that $\mathrm{Rank}(A)\leq M-1$, and so the DFT submatrix is not invertible.
\end{proof}

Note that our proof of Theorem~\ref{thm.uniformly distributed} establishes the necessity of having row indices uniformly distributed over the divisors of $N$ in the general case.
This leaves some hope for completely characterizing full spark harmonic frames.
Naturally, one might suspect that the uniform distribution condition is sufficient in general, but this suspicion fails when $N=10$.
Indeed, the following DFT submatrix is singular despite the row indices being uniformly distributed over the divisors of $10$:
\begin{equation*}
(e^{-2\pi i mn/10})_{m\in\{0,1,3,4\},n\in\{0,1,2,6\}}.
\end{equation*}

Also, just as we used Chebotar\"{e}v's theorem to analyze the harmonic equiangular tight frames from Xia et al.~\cite{XiaZG:it05}, we can also use Theorem~\ref{thm.uniformly distributed} to determine whether harmonic equiangular tight frames with a prime power number of frame elements are full spark.
Unfortunately, none of the infinite families in~\cite{XiaZG:it05} have the number of frame elements in the form of a prime power (other than primes).
Luckily, there is at least one instance in which the number of frame elements happens to be a prime power: the harmonic frames that arise from Singer difference sets have 
$M=\frac{q^d-1}{q-1}$ and $N=\frac{q^{d+1}-1}{q-1}$ for a prime power $q$ and an integer $d\geq2$; when $q=3$ and $d=4$, the number of frame elements $N=11^2$ is a prime power.
In this case, the row indices we select are
\begin{align*}
\mathcal{M}
=&\{1, 2, 3, 6, 7, 9, 11, 18, 20, 21, 25, 27, 33, 34, 38, 41, 44, 47, 53, 54, 55, 56, \\
& ~~~  58, 59, 60, 63, 64, 68, 70, 71, 75, 81, 83, 89, 92, 99, 100, 102, 104, 114\},
\end{align*}
but these are not uniformly distributed over 11, and so the corresponding harmonic frame is not full spark by Theorem~\ref{thm.uniformly distributed}.

\section{Full spark Parseval frames are dense}

Recently, Lu and Do~\cite{LuD:08} showed that full spark frames are dense in the entire set of matrices.
This corresponds to our intuition that a matrix whose entries are independent continuous random variables is full spark with probability one, which was also recently proved by Blumensath and Davies~\cite{BlumensathD:09}.
By contrast, as noted by Gorodnitsky and Rao~\cite{GorodnitskyR:97}, certain classes of frames which arise in practice, such as in physical tomography, are never full spark.
This issue also occurs in frame theory: Steiner ETFs form one of the largest known classes of ETFs, and yet none of them are full spark~\cite{FickusMT:laa11}.
As such, Bourguignon et al.~\cite{BourguignonCI:07} were prompted to prove that, among all Vandermonde frames with bases in the complex unit circle, full spark frames are dense.
In this section, we consider a very important class of frames, namely, those which exhibit Parseval tightness, where the frame bound is $1$.
Specifically, we show that full spark Parseval frames are dense in the entire set of Parseval frames.
Unlike the previous work in this vein, our techniques exploit general concepts in algebraic geometry, lending themselves to future application in proving further density results.

In order to make our arguments rigorous, we view each entry $F_{mn}$ of the $M\times N$ matrix $F$ in terms of its real and imaginary parts: $x_{mn}+iy_{mn}$; this decomposition will become helpful later when we consider inner products between rows of $F$, which are not algebraic operations on complex vectors.
Recall that $F$ is full spark precisely when each $M\times M$ submatrix has nonzero determinant.
We note that for each submatrix, the determinant is a polynomial in the $x_{mn}$'s and $y_{mn}$'s, and having this polynomial be nonzero is equivalent to having either its real or imaginary part be nonzero.
This naturally leads us to the following definition from algebraic geometry:
A \emph{real algebraic variety} is the set of common zeros of a finite set of polynomials, that is, given polynomials $p_1,\ldots,p_r\in\mathbb{R}[x_1,\ldots,x_k]$, we define the corresponding real algebraic variety by
\begin{equation*}
V(p_1,\ldots,p_r):=\{x\in\mathbb{R}^k:p_1(x)=\cdots=p_r(x)=0\}.
\end{equation*}
Each submatrix determinant corresponds to a real algebraic variety $V\subseteq\mathbb{R}^{2MN}$ of two polynomials, and having this determinant be nonzero is equivalent to restricting to the complement of $V$.

In general, a variety is equipped with a topology known as the \emph{Zariski topology}, in which subvarieties are the closed sets.
As such, the set of matrices with a nonzero determinant forms a Zariski-open set, since it is the complement of a variety.
In order to exploit this Zariski-openness, we require the additional concept of irreducibility: 
A variety is said to be \emph{irreducible} if it cannot be written as a finite union of proper subvarieties.  
As an example, the entire space $\mathbb{R}^k$ is the variety which corresponds to the zero polynomial; in this case, every proper subvariety is lower-dimensional, and so $\mathbb{R}^k$ is trivially irreducible.
On the other hand, the variety in $\mathbb{R}^2$ defined by $xy=0$ is not irreducible because it can be expressed as the union of varieties defined by $x=0$ and $y=0$.
Irreducibility is important because it says something about Zariski-open sets: 

\begin{thm}
\label{thm.irreducible dense}
If $V$ is an irreducible algebraic variety, then every nonempty Zariski-open subset of $V$ is dense in $V$ in the standard topology.
\end{thm}

For example, the hyperplane defined by $x_1=0$ is a subvariety of $\mathbb{R}^k$, and since $\mathbb{R}^k$ is irreducible, the complement of the hyperplane is dense in the standard topology.
Going back to the variety defined by $xy=0$, we know it can be expressed as a union of proper subvarieties, namely the $x$- and $y$-axes, and complementing one gives a subset of the other, neither of which is dense in the entire variety.

We are now ready to prove the following result:

\begin{thm}
\label{thm.overall density}
Every matrix is arbitrarily close to a full spark frame.
\end{thm}

\begin{proof}
In an $M\times N$ matrix, there are $\binom{N}{M}$ submatrices of size $M\times M$, and the determinant of each of these submatrices corresponds to a variety of two real polynomials.
Since the set of $M\times N$ full spark frames is defined as the (finite) intersection of the complements of these varieties, it is a Zariski-open subset of the irreducible variety $\mathbb{R}^{2MN}$ of all $M\times N$ matrices.
Moreover, this set is nonempty since it contains the matrix formed by the first $M$ rows of the $N\times N$ DFT, and so we are done by Theorem~\ref{thm.irreducible dense}.
\end{proof}

We now focus on the set of \emph{Parseval frames}, that is, tight frames with frame bound $1$.
Note that $M\times N$ Parseval frames are characterized by the rows forming an orthonormal system of size $M$ in $N$-dimensional space.
The set of all such orthonormal systems is known as the \emph{Stiefel manifold}, denoted $\mathrm{St}(M,N)$.
In general, a \emph{manifold} is a set of vectors with a well-defined tangent space at every point in the set.
Since this manifold property is nice, we would like to think of varieties as manifolds, but there exist varieties with points at which tangent spaces are not well-defined; such points are called \emph{singularities}.
Note that our definition of singularity is geometric, i.e., where the variety fails to be a real differentiable manifold, as opposed to algebraic.
For example, the variety defined by $xy=0$ is the union of the $x$- and $y$-axes, and as such, has a singularity at $x=y=0$.
Certainly with different types of varieties, there are other types of singularities which may arise, but there are also many varieties which do not have singularities at all, and we call these varieties \emph{nonsingular}.
Nonsingularity is a useful property for the following reason:

\begin{thm}
\label{thm.irreducible}
An algebraic variety which is nonsingular and connected is necessarily irreducible.
\end{thm}

This result follows from Theorem I.5.1 and Remark III.7.9.1 of Hartshorne~\cite{Hartshorne:book77}, which assumes that the variety is over an algebraically closed field, unlike $\mathbb{R}$; however, the proof is unaffected when removing the algebraically-closed assumption.

Note that the orthonormality conditions which characterize Parseval frames $F$ can be expressed as polynomial equations in the real and imaginary parts of the entries $F_{mn}$.
As such, we may view the Stiefel manifold as a real algebraic variety; this variety is nonsingular because it is a manifold.
Moreover, the variety is connected because the (connected) unitary group $\mathrm{U}(N)$ acts transitively on $\mathrm{St}(M,N)$.
Thus by Theorem~\ref{thm.irreducible}, the variety of Parseval frames is irreducible.
Having established this, we can now prove the following result:

\begin{thm}
\label{thm.parseval density}
Every Parseval frame is arbitrarily close to a full spark Parseval frame.
\end{thm}

\begin{proof}
Proceeding as in the proof of Theorem~\ref{thm.overall density}, we have that the set of $M\times N$ full spark Parseval frames is Zariski-open in the irreducible variety of all $M\times N$ Parseval frames.
Again considering the first $M$ rows of the $N\times N$ DFT, we know this set is nonempty, and so we are done by Theorem~\ref{thm.irreducible dense}.
\end{proof}

We note that Theorems~\ref{thm.overall density} and~\ref{thm.parseval density} are also true when we further require the frames to be real.
In this case, we cannot use the first $M$ rows of the $N\times N$ DFT to establish that the Zariski-open sets are nonempty.
For the real version of Theorem~\ref{thm.overall density}, we use an $M\times N$ Vandermonde matrix with distinct real bases; see Lemma~\ref{lem.vandermonde}.
However, this construction must be modified to use it in the proof of the real version of Theorem~\ref{thm.parseval density}, since such Vandermonde matrices will not be tight; see Theorem~\ref{thm.vandermonde}.
Given an $M\times N$ full spark frame $F$, the modification $G:=(FF^*)^{-1/2}F$ is full spark and Parseval; indeed, $GG^*=I_M$, and since $F$ is a frame, $(FF^*)^{-1/2}$ is full rank, and so the columns of $G$ are linearly independent precisely when the corresponding columns of $F$ are linearly independent.

Another way that the proof of Theorem~\ref{thm.parseval density} changes in the real case is in verifying that the real Stiefel manifold is irreducible.
Just as in the complex case, this follows from Theorem~\ref{thm.irreducible} since $\mathrm{St}(M,N)$ is connected~\cite{Rapcsak:ejor02}, but the fact that $\mathrm{St}(M,N)$ is connected in the real case is not immediate.
By analogy, the orthogonal group $\mathrm{O}(N)$ certainly acts transitively on $\mathrm{St}(M,N)$, but unlike the unitary group, the orthogonal group has two connected components.
Intuitively, this is resolved by the fact that $N>M$, granting additional freedom of movement throughout $\mathrm{St}(M,N)$.

In addition to Theorems~\ref{thm.overall density} and~\ref{thm.parseval density}, we would like a similar result for unit norm tight frames, i.e., that every unit norm tight frame is arbitrarily close to a full spark unit norm tight frame.
Certainly, the set of unit norm tight frames is a real algebraic variety, but it is unclear whether this variety is irreducible.
Without knowing whether the variety is irreducible, we can follow the proofs of Theorems~\ref{thm.overall density} and~\ref{thm.parseval density} to conclude that full spark unit norm tight frames are dense in the irreducible components in which they exist---a far cry from the density result we seek.
This illustrates a significant gap in our current understanding of the variety of unit norm tight frames.
It should be mentioned that Strawn~\cite{Strawn:jfaa11} showed that the variety of $M\times N$ unit norm tight frames (over real or complex space) is nonsingular precisely when $M$ and $N$ are relatively prime.
Additionally, Dykema and Strawn~\cite{DykemaS:ijpam06} proved that the variety of $2\times N$ real unit norm tight frames is connected, and so by Theorem~\ref{thm.irreducible}, this variety is irreducible when $N$ is odd.
It is unknown whether the variety of $M\times N$ unit norm tight frames is connected in general.
Finally, we note that a Theorem~\ref{thm.parseval density} gives a weaker version of the result we would like: every unit norm tight frame is arbitrarily close to a full spark tight frame with frame element lengths arbitrarily close to $1$.

\section{The computational complexity of verifying full spark}

In the previous section, we demonstrated the abundance of full spark frames, even after imposing the additional condition of tightness.
But how much computation is required to check whether a particular frame is full spark?
At the heart of the matter is computational complexity theory, which provides a rigorous playing field for expressing how hard certain problems are.
In this section, we consider the complexity of the following problem:

\begin{prob}[\textsc{Full Spark}]
\label{prob.full spark}
Given a matrix, is it full spark?
\end{prob}

For the lay mathematician, \textsc{Full Spark} is ``obviously'' $\NP$-hard because the easiest way he can think to solve it for a given $M\times N$ matrix is by determining whether each of the $M\times M$ submatrices is invertible; computing $\binom{N}{M}$ determinants would do, but this would take a lot of time, and so \textsc{Full Spark} must be $\NP$-hard.
However, computing $\binom{N}{M}$ determinants may not necessarily be the fastest way to test whether a matrix is full spark.
For example, perhaps there is an easy-to-calculate expression for the product of the determinants; after all, this product is nonzero precisely when the matrix is full spark.
Recall that Theorem~\ref{thm.uniformly distributed} gives a very straightforward litmus test for \textsc{Full Spark} in the special case where the matrix is formed by rows of a DFT of prime-power order---who's to say that a version of this test does not exist for the general case?
If such a test exists, then it would suffice to find it, but how might one disprove the existence of any such test?
Indeed, since we are concerned with the necessary amount of computation, as opposed to a sufficient amount, the lay mathematician's intuition is a bit misguided.

To discern how much computation is necessary, the main feature of interest is a problem's \emph{complexity}.
We use complexity to compare problems and determine whether one is harder than the other.
As an example of complexity, intuitively, doubling an integer is no harder than adding integers, since one can use addition to multiply by $2$; put another way, the complexity of doubling is somehow ``encoded'' in the complexity of adding, and so it must be lesser (or equal).
To make this more precise, complexity theorists use what is called a \emph{polynomial-time reduction}, that is, a polynomial-time algorithm that solves problem $A$ by exploiting an oracle which solves problem $B$; the reduction indicates that solving problem $A$ is no harder than solving problem $B$ (up to polynomial factors in time), and we say ``$A$ reduces to $B$,'' or $A\leq B$.
Since we can use the polynomial-time routine $x+x$ to produce $2x$, we conclude that doubling an integer reduces to adding integers, as expected.

In complexity theory, problems are categorized into complexity classes according to the amount of resources required to solve them.
For example, the complexity class $\P$ contains all problems which can be solved in polynomial time, while problems in $\EXP$ may require as much as exponential time.
Problems in $\NP$ have the defining quality that solutions can be verified in polynomial time given a certificate for the answer.
As an example, the graph isomorphism problem is in $\NP$ because, given an isomorphism between graphs (a certificate), one can verify that the isomorphism is legit in polynomial time.
Clearly, $\P\subseteq\NP$, since we can ignore the certificate and still solve the problem in polynomial time.
Finally, a problem $B$ is called $\NP$-\emph{hard} if every problem $A$ in $\NP$ reduces to $B$, and a problem is called $\NP$-\emph{complete} if it is both $\NP$-hard and in $\NP$.
In plain speak, $\NP$-hard problems are harder than every problem in $\NP$, while $\NP$-complete problems are the hardest of problems in $\NP$.

At this point, it should be clear that $\NP$-hard problems are not merely problems that seem to require a lot of computation to solve.
Certainly, $\NP$-hard problems have this quality, as an $\NP$-hard problem can be solved in polynomial time only if $\P=\NP$; this is an open problem, but it is 
widely believed that $\P\neq\NP$.
However, there are other problems which seem hard but are not known to be $\NP$-hard (e.g., the graph isomorphism problem).
Rather, to determine whether a problem is $\NP$-hard, one must find a polynomial-time reduction that compares the problem to all problems in $\NP$.
To this end, notice that $A\leq B$ and $B\leq C$ together imply $A\leq C$, and so to demonstrate that a problem $C$ is $\NP$-hard, it suffices to show that $B\leq C$ for some $\NP$-hard problem $B$.

Unfortunately, it can sometimes be difficult to find a deterministic reduction from one problem to another.
One example is reducing the satisfiability problem (\textsc{SAT}) to the unique satisfiability problem (\textsc{Unique~SAT}).
To be clear, \textsc{SAT} is an $\NP$-hard problem~\cite{Karp:ccc72} that asks whether there exists an input for which a given Boolean function returns ``true,'' while \textsc{Unique~SAT} asks the same question with an additional promise: that the given Boolean function is satisfiable only if there is a \emph{unique} input for which it returns ``true.''
Intuitively, \textsc{Unique~SAT} is easier than \textsc{SAT} because we might be able to exploit the additional structure of uniquely satisfiable Boolean functions; thus, it could be difficult to find a reduction from \textsc{SAT} to \textsc{Unique~SAT}.
Despite this intuition, there is a \emph{randomized} polynomial-time reduction from \textsc{SAT} to \textsc{Unique~SAT}~\cite{ValiantV:tcs86}.
Defined over all Boolean functions of $n$ variables, the reduction maps functions that are not satisfiable to other functions that are not satisfiable, and with probability $\geq\frac{1}{8n}$, it maps satisfiable functions to uniquely satisfiable functions.
After applying this reduction to a given Boolean function, if a \textsc{Unique~SAT} oracle declares ``uniquely satisfiable,'' then we know for certain that the original Boolean function was satisfiable.
But the reduction will only map a satisfiable problem to a uniquely satisfiable problem with probability $\geq\frac{1}{8n}$, so what good is this reduction?
The answer lies in something called \emph{amplification}; since the success probability is, at worst, polynomially small in $n$ (i.e., $\geq\frac{1}{p(n)}$), we can repeat our oracle-based randomized algorithm a polynomial number of times $np(n)$ and achieve an error probability $\leq(1-\frac{1}{p(n)})^{np(n)}\sim e^{-n}$ which is exponentially small.

In this section, we give a randomized polynomial-time reduction from a problem in matroid theory.
Before stating the problem, we first briefly review some definitions.
To each bipartite graph with bipartition $(E,E')$, we associate a \emph{transversal matroid} $(E,\mathcal{I})$, where $\mathcal{I}$ is the collection of subsets of $E$ whose vertices form the ends of a matching in the bipartite graph; subsets in $\mathcal{I}$ are called $\emph{independent}$.
Hall's marriage theorem \cite{Hall:jlms35} gives a remarkable characterization of the independent sets in a transversal matroid: $B\in\mathcal{I}$ if and only if every subset $A\subseteq B$ has $\geq|A|$ neighbors in the bipartite graph.
Next, just as spark is the size of the smallest linearly dependent set, the \emph{girth} of a matroid is the size of the smallest subset of $E$ that is not in $\mathcal{I}$.
In fact, this analogy goes deeper: 
A matroid is \emph{representable over a field} $\mathbb{F}$ if, for some $M$, there exists a mapping $\varphi\colon E\rightarrow\mathbb{F}^M$ such that $\varphi(A)$ is linearly independent if and only if $A\in\mathcal{I}$; as such, the girth of $(E,\mathcal{I})$ is the spark of $\varphi(E)$.
In our reduction, we make use of the fact that every transversal matroid is representable over $\mathbb{R}$~\cite{PiffW:jlms70}.
We are now ready to state the problem from which we will reduce \textsc{Full Spark}:

\begin{prob}
\label{prob.girth}
Given a bipartite graph, what is the girth of its transversal matroid?
\end{prob}

Before giving the reduction, we will show that Problem~\ref{prob.girth} is $\NP$-hard.
The result comes from McCormick's thesis~\cite{McCormick:phd83}, which credits the proof to Stockmeyer; since~\cite{McCormick:phd83} is difficult to access and the proof is instructive, we include it below:

\begin{thm}
Problem~\ref{prob.girth} is $\NP$-hard.
\end{thm}

\begin{proof}
We will reduce from the $\NP$-complete clique decision problem, which asks ``Given a graph, does it contain a clique of $K$ vertices?''~\cite{Karp:ccc72}.
First, we may assume $K\geq 4$ without loss of generality, since any such clique can otherwise be found in cubic time by an exhaustive search.
Take a graph $G=(V,E)$, and consider the bipartite graph $G'$ between disjoint sets $E$ and $V\sqcup\{1,\ldots,\binom{K}{2}-K-1\}$, in which $e\leftrightarrow v$ for every $e\in E$ and $v\in e$, and $e\leftrightarrow k$ for every $e\in E$ and $k\in\{1,\ldots,\binom{K}{2}-K-1\}$.
We claim that the girth of the transversal matroid of $G'$ is $\binom{K}{2}$ precisely when there exists a $K$-clique in $G$.

We start by analyzing the girth of a transversal matroid.
Consider any dependent set $C\subseteq E$ with $\leq|C|-2$ neighbors in $G'$.
Then removing any member $x$ of $C$ will produce a smaller set $C\setminus\{x\}$ with $\leq|C|-2=|C\setminus\{x\}|-1$ neighbors in $G'$, which is necessarily dependent by the pigeonhole principle.
Now consider any dependent set $C\subseteq E$ with $\geq|C|$ neighbors in $G'$.
By the Hall's marriage theorem, there exists a proper subset $C'\subseteq C$ with $<|C'|$ neighbors in $G'$, meaning $C'$ is a smaller dependent set.
Thus, the girth $c$ is the size of the smallest subset $C\subseteq E$ with $|C|-1$ total neighbors in $G'$.

Suppose $c=\binom{K}{2}$.
Then since $C$ is adjacent to every vertex in $\{1,\ldots,\binom{K}{2}-K-1\}$, $C$ has 
$(c-1)-(\binom{K}{2}-K-1)=K$ neighbors in $V$.
These are precisely the vertices $D$ in $G$ which are induced by the edges in $C$, and so $C$ is contained in the set $C'$ of edges induced by $D$, of which there are $\leq\binom{|D|}{2}=\binom{K}{2}$, with equality only if $D$ induces a $K$-clique in $G$.
Since $\binom{K}{2}=|C|\leq|C'|\leq\binom{|D|}{2}=\binom{K}{2}$ implies equality, there exists a $K$-clique in $G$.

Now suppose there exists a $K$-clique with edges $C$.
Then $C$ has $\binom{K}{2}$ elements and $K+(\binom{K}{2}-K-1)=\binom{K}{2}-1$ neighbors in $G'$.
To prove that $c=\binom{K}{2}$, it suffices to show that there is no smaller subset $C'\subseteq E$ with $|C'|-1$ total neighbors in $G'$.
Suppose, to the contrary, that $c=\binom{K}{2}-\ell$ for some $\ell>0$.
Then there exists $C'\subseteq E$ with $c$ elements and $(c-1)-(\binom{K}{2}-K-1)=K-\ell$ neighbors in $V$.
Note that each $e\in E$ contains two vertices in $G$, and so by the definition of $G'$, $C'$ necessarily has $K-\ell\geq 2$ neighbors in $V$.
Also, since the $K-\ell$ neighbors of $C'$ in $V$ arise from the subgraph of $G$ induced by $C'$, and since those neighbors induce at most $\binom{K-\ell}{2}$ edges including $C'$, we have $\binom{K}{2}-\ell=c\leq\binom{K-\ell}{2}$.
This inequality simplifies to $\ell\geq 2K-3$, which combines with $K-\ell\geq 2$ to contradict the fact that $\ell>0$.
\end{proof}

Having established that Problem~\ref{prob.girth} is $\NP$-hard, we reduce from it the main problem of this section.
Our proof is specifically geared toward the case where the matrix in question has integer entries; this is stronger than manipulating real (complex) numbers exactly as well as with truncations and tolerances.

\begin{thm}
\label{thm.full spark hard}
\textsc{Full Spark} is hard for $\NP$ under randomized polynomial-time reductions.
\end{thm}

\begin{proof}
We will give a randomized polynomial-time reduction from Problem~\ref{prob.girth} to \textsc{Full Spark}.
As such, suppose we are given a bipartite graph $G$, in which every edge is between the disjoint sets $A$ and $B$.
Take $M:=|B|$ and $N:=|A|$.
Using this graph, we randomly draw an $M\times N$ matrix $F$ using the following process: for each $i\in B$ and $j\in A$, pick the entry $F_{ij}$ randomly from $\{1,\ldots,N2^{N+1}\}$ if $i\leftrightarrow j$ in $G$; otherwise set $F_{ij}=0$.  
In Proposition 3.11 of~\cite{Marx:tcs09}, it is shown that the columns of $F$ form a representation of the transversal matroid of $G$ with probability $\geq\frac{1}{2}$.
For the moment, we assume that $F$ succeeds in representing the matroid.

Since the girth of the original matroid equals the spark of its representation, for each $K=1,\ldots,M$, we test whether $\mathrm{Spark}(F)>K$.
To do this, take $H$ to be some $M\times P$ full spark frame.
We will determine an appropriate value for $P$ later, but for simplicity, we can take $H$ to be the Vandermonde matrix formed from bases $\{1,\ldots,P\}$; see Lemma~\ref{lem.vandermonde}.
We claim we can randomly select $K$ indices $\mathcal{K}\subseteq\{1,\ldots,P\}$ and test whether $H_\mathcal{K}^*F$ is full spark to determine whether $\mathrm{Spark}(F)>K$.
Moreover, after performing this test for each $K=1,\ldots,M$, the probability of incorrectly determining $\mathrm{Spark}(F)$ is $\leq\frac{1}{2}$, provided $P$ is sufficiently large.

We want to test whether $H_\mathcal{K}^*F$ is full spark and use the result as a proxy for whether $\mathrm{Spark}(F)>K$.
For this to work, we need to have $\mathrm{Rank}(H_\mathcal{K}^*F_{\mathcal{K}'})=K$ precisely when $\mathrm{Rank}(F_{\mathcal{K}'})=K$ for every $\mathcal{K}'\subseteq\{1,\ldots,N\}$ of size $K$.
To this end, it suffices to have the nullspace $\mathcal{N}(H_\mathcal{K}^*)$ of $H_\mathcal{K}^*$ intersect trivially with the column space of $F_{\mathcal{K}'}$ for every $\mathcal{K}'$.
To be clear, it is always the case that $\mathrm{Rank}(H_\mathcal{K}^*F_{\mathcal{K}'})\leq\mathrm{Rank}(F_{\mathcal{K}'})$, and so $\mathrm{Rank}(F_{\mathcal{K}'})<K$ implies $\mathrm{Rank}(H_\mathcal{K}^*F_{\mathcal{K}'})<K$.
If we further assume that $\mathcal{N}(H_\mathcal{K}^*)\cap\mathrm{Span}(F_{\mathcal{K}'})=\{0\}$, then the converse also holds.
To see this, suppose $\mathrm{Rank}(H_\mathcal{K}^*F_{\mathcal{K}'})<K$.
Then by the rank-nullity theorem, there is a nontrivial $x\in\mathcal{N}(H_\mathcal{K}^*F_{\mathcal{K}'})$.
Since $H_\mathcal{K}^*F_{\mathcal{K}'}x=0$, we must have $F_{\mathcal{K}'}x\in\mathcal{N}(H_\mathcal{K}^*)$, which in turn implies $x\in\mathcal{N}(F_{\mathcal{K}'})$ since $\mathcal{N}(H_\mathcal{K}^*)\cap\mathrm{Span}(F_{\mathcal{K}'})=\{0\}$ by assumption.
Thus, $\mathrm{Rank}(F_{\mathcal{K}'})<K$ by the rank-nullity theorem.

Now fix $\mathcal{K}'\subseteq\{1,\ldots,N\}$ of size $K$ such that $\mathrm{Rank}(F_{\mathcal{K}'})=K$.
We will show that the vast majority of choices $\mathcal{K}\subseteq\{1,\ldots,P\}$ of size $K$ satisfy $\mathcal{N}(H_\mathcal{K}^*)\cap\mathrm{Span}(F_{\mathcal{K}'})=\{0\}$.
To do this, we consider the columns $\{h_k\}_{k\in\mathcal{K}}$ of $H_\mathcal{K}$ one at a time, and we make use of the fact that $\mathcal{N}(H_\mathcal{K}^*)=\bigcap_{k\in\mathcal{K}}\mathcal{N}(h_k^*)$.
In particular, since $H$ is full spark, there are at most $M-K$ columns of $H$ in the orthogonal complement of $\mathrm{Span}(F_{\mathcal{K}'})$, and so there are at least $P-(M-K)$ choices of $h_{k_1}$ for which $\mathcal{N}(h_{k_1}^*)$ does not contain $\mathrm{Span}(F_{\mathcal{K}'})$, i.e.,
\begin{equation*}
\mathrm{dim}\Big(\mathcal{N}(h_{k_1}^*)\cap\mathrm{Span}(F_{\mathcal{K}'})\Big)=K-1.
\end{equation*}
Similarly, after selecting the first $J$ $h_k$'s, we have $\mathrm{dim}(S)=K-J$, where
\begin{equation*}
S:=\bigcap_{j=1}^J\mathcal{N}(h_{k_j}^*)\cap\mathrm{Span}(F_{\mathcal{K}'}).
\end{equation*}
Again, since $H$ is full spark, there are at most $M-(K-J)$ columns of $H$ in the orthogonal complement of $S$, and so the remaining $P-(M-(K-J))$ columns are candidates for $h_{k_{J+1}}$ that give
\begin{equation*}
\mathrm{dim}\bigg(\bigcap_{j=1}^{J+1}\mathcal{N}(h_{k_j}^*)\cap\mathrm{Span}(F_{\mathcal{K}'})\bigg)
=\mathrm{dim}\Big(\mathcal{N}(h_{k_{J+1}}^*)\cap S\Big)
=K-(J+1).
\end{equation*}
Overall, if we randomly pick $\mathcal{K}\subseteq\{1,\ldots,P\}$ of size $K$, then
\begin{align*}
\mathrm{Pr}\Big(\mathcal{N}(H_\mathcal{K}^*)\cap\mathrm{Span}(F_{\mathcal{K}'})=\{0\}\Big)
&\geq(1-\tfrac{M-K}{P})(1-\tfrac{M-(K-1)}{P})\cdots(1-\tfrac{M-1}{P})\\
&\geq(1-\tfrac{M}{P})^K\\
&\geq 1-\tfrac{MK}{P},
\end{align*}
where the final step is by Bernoulli's inequality.
Taking a union bound over all choices of $\mathcal{K}'\subseteq\{1,\ldots,N\}$ and all values of $K=1,\ldots,M$ then gives
\begin{align*}
\mathrm{Pr}\bigg(\begin{array}{c}\mbox{fail to determine}\\\mbox{$\mathrm{Spark}(F)$}\end{array}\bigg)
&\leq\sum_{K=1}^M\binom{N}{K}\mathrm{Pr}\Big(\mathcal{N}(H_\mathcal{K}^*)\cap\mathrm{Span}(F_{\mathcal{K}'})\neq\{0\}\Big)\\
&\leq\sum_{K=1}^M\binom{N}{K}\frac{MK}{P}\\
&\leq\frac{M^32^N}{P}.
\end{align*}
Thus, to make the probability of failure $\leq\frac{1}{2}$, it suffices to have $P=M^32^{N+1}$.

In summary, we succeed in representing the original matroid with probability $\geq\frac{1}{2}$, and then we succeed in determining the spark of its representation with probability $\geq\frac{1}{2}$.
The probability of overall success is therefore $\geq\frac{1}{4}$.
Since our success probability is, at worst, polynomially small, we can apply amplification to achieve an exponentially small error probability.
\end{proof}

Our use of random linear projections in the above reduction to \textsc{Full Spark} is similar in spirit to Valiant and Vazirani's use of random hash functions in their reduction to \textsc{Unique~SAT}~\cite{ValiantV:tcs86}.
Since their randomized reduction is the canonical example thereof, we find our reduction to be particularly natural.

As a final note, we clarify that Theorem~\ref{thm.full spark hard} is a statement about the amount of computation necessary in the \emph{worst case}.
Indeed, the hardness of \textsc{Full Spark} does not rule out the existence of smaller classes of matrices for which full spark is easily determined.
As an example, Theorem~\ref{thm.uniformly distributed} determines \textsc{Full Spark} in the special case where the matrix is formed by rows of a DFT of prime-power order.
This illustrates the utility of applying additional structure to efficiently solve the \textsc{Full Spark} problem, and indeed, such classes of matrices are rather special for this reason.


\begin{thebibliography}{00}

\bibitem{BajwaCM:acha11}
Bajwa, W.U., Calderbank, R., Mixon, D.G.:
Two are better than one: Fundamental parameters of frame coherence,
Appl. Comput. Harmon. Anal. (in press)

\bibitem{BalanBCE:spie07}
Balan, R., Bodmann, B.G., Casazza, P.G., Edidin, D.:
Fast algorithms for signal reconstruction without phase,
Proc. SPIE, 67011L, 1--9 (2007)

\bibitem{BalanBCE:jfaa09}
Balan, R., Bodmann, B.G., Casazza, P.G., Edidin, D.:
Painless reconstruction from magnitudes of frame coefficients,
J. Fourier Anal. Appl. \textbf{15}, 488--501 (2009)

\bibitem{BalanCE:acha06}
Balan, R., Casazza, P., Edidin, D.:
On signal reconstruction without phase,
Appl. Comput. Harmon. Anal. \textbf{20}, 345--356 (2006)

\bibitem{BlumensathD:09}
Blumensath, T., Davies, M.E.:
Sampling Theorems for Signals From the Union of Finite-Dimensional Linear Subspaces,
IEEE Trans. Inform. Theory \textbf{55}, 1872--1882 (2009)

\bibitem{BourguignonCI:07}
Bourguignon, S., Carfantan, H., Idier, J.:
A Sparsity-Based Method for the Estimation of Spectral Lines From Irregularly Sampled Data,
IEEE J. Sel. Topics Signal Process. \textbf{1}, 575--585 (2007)

\bibitem{CahillCH:sampta11}
Cahill, J., Casazza, P.G., Heinecke, A.:
A notion of redundancy for infinite frames, 
Proc. Sampl. Theory Appl. (2011) 

\bibitem{Candes:08}
Cand\`{e}s, E.J.:
The restricted isometry property and its implications for compressed sensing,
C. R. Acad. Sci. Paris, Ser. I \textbf{346}, 589--592 (2008)

\bibitem{CandesRT:06}
Cand\`{e}s, E.J., Romberg, J., Tao, T.:
Robust uncertainty principles: exact signal reconstruction from highly incomplete frequency information,
IEEE Trans. Inform. Theory \textbf{52}, 489--509 (2006)

\bibitem{CandesSV:arxiv11}
Cand\`{e}s, E.J., Strohmer, T., Voroninski, V.:
PhaseLift: Exact and stable signal recovery from magnitude measurements via convex programming.
Available online: arXiv:1109.4499

\bibitem{CandesT:it05}
Cand\`{e}s, E.J., Tao, T.:
Decoding by linear programming,
IEEE Trans. Inform. Theory \textbf{51}, 4203--4215 (2005)

\bibitem{CasazzaHKK:it}
Casazza, P.G., Heinecke, A., Krahmer, F., Kutyniok, G.:
Optimally sparse frames,
IEEE Trans. Inform. Theory \textbf{57}, 7279--7287 (2011)

\bibitem{CasazzaT:06}
Casazza, P.G., Tremain, J.C.:
The Kadison-Singer Problem in mathematics and engineering,
Proc. Natl. Acad. Sci. U.S.A. \textbf{103}, 2032--2039 (2006)


\bibitem{DavenportDEK:csta11}
Davenport, M.A., Duarte, M.F., Eldar, Y.C., Kutyniok, G.:
Introduction to Compressed Sensing.
In: Eldar, Y.C., Kutyniok, G. (Eds.),
Compressed Sensing: Theory and Applications, Cambridge University Press (2011)

\bibitem{DelvauxV:laa08}
Delvaux, S., Van Barel, M.:
Rank-deficient submatrices of Fourier matrices,
Linear Algebra Appl. \textbf{429}, 1587--1605 (2008)

\bibitem{DonohoE:pnas03}
Donoho, D.L., Elad, M.:
Optimally sparse representation in general (nonorthogonal) dictionaries via $\ell^1$ minimization,
Proc. Nat. Acad. Sci. \textbf{100}, 2197--2202 (2003)

\bibitem{DykemaS:ijpam06}
Dykema, K., Strawn, N.:
Manifold structure of spaces of spherical tight frames,
Int. J. Pure Appl. Math. \textbf{28}, 217--256 (2006)

\bibitem{EvansI:ams77}
Evans, R.J., Isaacs, I.M.:
Generalized Vandermonde determinants and roots of prime order,
Proc. Amer. Math. Soc. \textbf{58}, 51--54 (1977)

\bibitem{FickusM:spie11}
Fickus, M., Mixon, D.G.:
Deterministic matrices with the restricted isometry property,
Proc. SPIE (2011)

\bibitem{FickusMT:laa11}
Fickus, M., Mixon, D.G., Tremain, J.C.:
Steiner equiangular tight frames,
Linear Algebra Appl. (in press)

\bibitem{Fuchs:05}
Fuchs, J.-J.:
Sparsity and uniqueness for some specific under-determined linear systems,
Proc. IEEE Int. Conf. Acoust. Speech Signal Process., 729--732 (2005)

\bibitem{GorodnitskyR:97}
Gorodnitsky, I.F., Rao, B.D.:
Sparse signal reconstruction from limited data using FOCUSS: A re-weighted minimum norm algorithm,
IEEE Trans. Signal Process. \textbf{45}, 600--616 (1997)

\bibitem{Goyal:phd98}
Goyal, V.K.:
Beyond Traditional Transform Coding.
Ph.D. Thesis, University California, Berkeley (1998)

\bibitem{Hall:jlms35}
Hall, P.:
On Representatives of Subsets,
J. London Math. Soc. \textbf{10}, 26--30 (1935)

\bibitem{Hartshorne:book77}
Hartshorne, R.:
Algebraic Geometry.
Graduate Texts in Mathematics, Springer, New York (1977)

\bibitem{HolmesP:laa04}
Holmes, R.B., Paulsen, V.I.:
Optimal frames for erasures,
Linear Algebra Appl. \textbf{377}, 31--51 (2004)

\bibitem{JungnickelPS:hcd07}
Jungnickel, D., Pott, A., Smith, K.W.:
Difference Sets.
In: Colbourn, C.J., Dinitz, J.H. (Eds.), 
Handbook of Combinatorial Designs, 2nd ed., pp.~419--435 (2007)

\bibitem{Karp:ccc72}
Karp, R.M.:
Reducibility Among Combinatorial Problems.
In: Miller, R.E., Thatcher, J.W. (Eds.),
Complexity of Computer Computations,
pp.~85--103. Plenum, New York (1972)

\bibitem{LuD:08}
Lu, Y.M., Do, M.N.:
A Theory for Sampling Signals From a Union of Subspaces,
IEEE Trans. Signal Process. \textbf{56}, 2334--2345 (2008)

\bibitem{Marx:tcs09}
Marx, D.:
A parameterized view on matroid optimization problems,
Theor. Comput. Sci. \textbf{410}, 4471--4479 (2009)

\bibitem{McCormick:phd83}
McCormick, S.T.:
A Combinatorial Approach to Some Sparse Matrix Problems.
Ph.D. Thesis, Stanford University (1983)

\bibitem{MixonQKF:icassp11}
Mixon, D.G., Quinn, C., Kiyavash, N., Fickus, M.:
Equiangular tight frame fingerprinting codes,
Proc. IEEE Int. Conf. Acoust. Speech Signal Process., 1856--1859 (2011)

\bibitem{MohimaniBJ:09}
Mohimani, H., Babaie-Zadeh, M., Jutten, C.:   
A Fast Approach for Overcomplete Sparse Decomposition Based on Smoothed $\ell^0$ Norm,
IEEE Trans. Signal Process. \textbf{57}, 289--301 (2009)

\bibitem{NakamuraM:ieeetc82}
Nakamura, S., Masson, G.M.:   
Lower bounds on crosspoints in concentrators,
IEEE Trans. Comput. \textbf{C-31}, 1173--1179 (1982)

\bibitem{PiffW:jlms70}
Piff, M.J., Welsh, D.J.A.:
On the vector representation of matroids,
J. London Math. Soc. \textbf{2}, 284--288 (1970)

\bibitem{PuschelK:dcc05}
P\"{u}schel, M., Kova\v{c}evi\'{c}, J.: 
Real, tight frames with maximal robustness to erasures,
Proc. Data Compr. Conf., 63--72 (2005)

\bibitem{Rapcsak:ejor02}
Rapcs\'{a}k, T.:
On minimization on Stiefel manifolds,
Eur. J. Oper. Res. \textbf{143}, 365--376 (2002)

\bibitem{Renes:laa07}
Renes, J.:
Equiangular tight frames from Paley tournaments,
Linear Algebra Appl. \textbf{426}, 497--501 (2007)

\bibitem{RenesBSC:jmp04}
Renes, J.M., Blume-Kohout, R., Scott, A.J., Caves, C.M.:
Symmetric informationally complete quantum measurements,
J. Math. Phys. \textbf{45}, 2171--2180 (2004)

\bibitem{RudelsonV:cpaa08}
Rudelson, M., Vershynin, R.:
On sparse reconstruction from Fourier and Gaussian measurements,
Commun. Pure Appl. Anal. \textbf{61}, 1025--1045 (2008)

\bibitem{StevenhagenL:mi96}
Stevenhagen, P., Lenstra, H.W.:
Chebotar\"{e}v and his density theorem,
Math. Intelligencer \textbf{18}, 26--37 (1996)

\bibitem{Strawn:jfaa11}
Strawn, N.:
Finite frame varieties: Nonsingular points, tangent spaces, and explicit local parameterizations,
J. Fourier Anal. Appl. \textbf{17}, 821--853 (2011)

\bibitem{Strohmer:laa08}
Strohmer, T.:
A note on equiangular tight frames,
Linear Algebra Appl. \textbf{429}, 326--330 (2008)

\bibitem{StrohmerH:acha03}
Strohmer, T., Heath, R.W.:
Grassmannian frames with applications to coding and communication,
Appl. Comput. Harmon. Anal. \textbf{14}, 257--275 (2003)

\bibitem{TangN:10}
Tang, G., Nehorai, A.:
Performance Analysis for Sparse Support Recovery,
IEEE Trans. Inform. Theory \textbf{56}, 1383--1399 (2010)

\bibitem{Tao:mrl05}
Tao, T.:
An uncertainty principle for cyclic groups of prime order,
Math. Research Letters \textbf{12}, 121--128 (2005)

\bibitem{Tropp:acha08}
Tropp, J.A.:
On the conditioning of random subdictionaries,
Appl. Comput. Harmon. Anal. \textbf{25}, 1--24 (2008)

\bibitem{ValiantV:tcs86}
Valiant, L., Vazirani, V.:
NP is as easy as detecting unique solutions,
Theor. Comput. Sci. \textbf{47}, 85--93 (1986)

\bibitem{XiaZG:it05}
Xia, P., Zhou, S., Giannakis, G.B.:   
Achieving the Welch bound with difference sets,
IEEE Trans. Inform. Theory \textbf{51}, 1900--1907 (2005)

\bibitem{WipfR:04}
Wipf, D.P., Rao, B.D.:
Sparse Bayesian learning for basis selection,
IEEE Trans. Signal Process. \textbf{52}, 2153--2164 (2004)

\bibitem{Zauner:phd99}
Zauner, G.:
Quantendesigns: Grundz\"{u}ge einer nichtkommutativen Designtheorie.
Ph.D. thesis, University of Vienna (1999)

\end{thebibliography}
\end{document}